\documentclass[sn-mathphys,Numbered]{sn-jnl}
\pdfoutput=1
\usepackage{pgfplots}
\usepackage{graphicx,epstopdf}%
\usepackage{multirow}%
\usepackage{amsmath,amssymb,amsfonts}%
\usepackage{amsthm}%
\usepackage{mathrsfs}%
\usepackage[title]{appendix}%
\usepackage{xcolor}%
\usepackage{textcomp}%
\usepackage{manyfoot}%
\usepackage{booktabs}%
\usepackage{algpseudocode}%
\usepackage{listings}%
\usepackage{ragged2e}
\usepackage{fullpage}
\usepackage{changepage}
\usepackage{stfloats}
\usepackage[caption=false]{subfig}
\usepackage{extarrows}
\usepackage{cleveref}
\usepackage{tikz}
\usepackage{pifont}
\usepackage{enumitem}
\usepackage[ruled, vlined, noline]{algorithm2e}
\usepackage{float}
\usepackage{graphicx}
\usepackage{subcaption}
\usepackage{booktabs,tabularx,array}
\usepackage{bm}

\theoremstyle{thmstyleone}%
\newtheorem{theorem}{Theorem}[section]%  meant for continuous numbers
\newtheorem{proposition}[theorem]{Proposition}% 
\theoremstyle{thmstyletwo}%
\newtheorem{remark}{Remark}%
\newtheorem{corollary}{Corollary}[section]%
\theoremstyle{thmstylethree}%
\newtheorem{lemma}{Lemma}[section]
\begin{document}
	\title[Article Title]{Fast Rates and Strong Convergence of Tikhonov-Regularized Mixed-Order Primal--Dual Dynamics for Linearly Constrained Optimization Without Eventual Ball
		Conditions}
	
	\author[1]{\fnm{Hong-lu} \sur{Li}}\email{lihonglu$\_$7988@163.com}
	
	\author*[1]{\fnm{Yi-bin} \sur{Xiao}}\email{xiaoyb9999@hotmail.com}

	\affil[1]{\orgdiv{School of Mathematical Sciences}, \orgname{University of Electronic Science and Technology of China}, \orgaddress{\street{}\city{Chengdu}, \postcode{611731}, \state{Sichuan}, \country{P.R. China}}}

  \maketitle

\begin{abstract}

	    \textbf{Abstract.}
		In this paper, we study a Tikhonov-regularized mixed-order primal--dual dynamical system with implicit Hessian damping for linearly constrained convex optimization problems in finite-dimensional Euclidean spaces, where the primal equation is second order and incorporates the viscous damping term \(\delta\sqrt{\varepsilon(t)}\,\dot x(t)\), whereas the multiplier equation remains first order. By constructing a new class of energy functions, for a
		general Tikhonov regularization coefficient \(\varepsilon(t)\), we prove
		the strong convergence of the primal trajectory and derive fast convergence
		rates under the same parameter assumptions, without imposing
		any eventual inside/outside-ball condition. More precisely, the primal trajectory converges to the
		minimum-norm solution, and the multiplier converges to a compatible KKT
		multiplier, while the convergence rates of the Lagrangian gap, feasibility
		violation, and objective residual are \(o(\varepsilon(t))\), and the
		convergence rate of the velocity norm is \(o(\sqrt{\varepsilon(t)})\).	For the critical case \(\varepsilon(t)=c/t^2\), in which the damping coefficient
		\(\delta\sqrt{\varepsilon(t)}\) reduces to \(\delta\sqrt{c}/t\), we establish
		the sharper convergence rates \(o(t^{-2})\) for the Lagrangian gap,
		feasibility violation, and objective residual, together with \(o(t^{-1})\)
		for the velocity norm, which improve the corresponding \(O(t^{-2})\) and
		\(O(t^{-1})\) decay estimates obtained in the related literature. Most importantly, when the proposed dynamical system is specialized to the finite-dimensional unconstrained setting, our analysis answers the open question on strong convergence in this critical regime posed by Attouch and L\'aszl\'o [Math. Methods Oper. Res., 99 (2024), pp.~307--347].
		
		\noindent\textbf{Keywords.} Linearly constrained convex optimization, implicit Hessian damping, strong convergence, minimum-norm solution, \(o(t^{-2})\) convergence rates
\end{abstract}

	\section{Introduction}\label{sec1}
Let \(f:\mathbb{R}^n\to\mathbb{R}\) be convex and continuously
differentiable with an \(L\)-Lipschitz continuous gradient. Given
\(A\in\mathbb{R}^{m\times n}\) and \(b\in\mathbb{R}^m\), consider
\begin{equation}\label{problem}
	\min_{x\in\mathbb{R}^n} f(x)
	\qquad\text{subject to}\qquad Ax=b.
\end{equation}
This model arises in consensus and distributed optimization, resource
allocation, network flows, and signal and image reconstruction; see
\cite{Boyd ML 2011,Candes SPM 2008,Feijer Automatica 2010,
	Zhang SIIMS 2010,Zeng TAC 2023}. The unconstrained case corresponds
to \(A=0\) and \(b=0\). Define $\mathcal F:=\{x\in\mathbb{R}^n:Ax=b\},
S:=\operatorname*{argmin}_{x\in\mathcal F}f(x),$
and, for \(\rho\ge0\),
\[
\mathcal L_\rho(x,\lambda)
:=
f(x)+\langle\lambda,Ax-b\rangle
+\frac{\rho}{2}\Vert Ax-b\Vert^2.
\]
The primal--dual solution set is $\Omega:=
\left\{
(x,\lambda)\in\mathbb{R}^n\times\mathbb{R}^m:
Ax=b,\ \nabla f(x)+A^\top\lambda=0
\right\}.$
We assume throughout that \(\Omega\neq\emptyset\), equivalently
\(S\neq\emptyset\) in the present finite-dimensional affine setting.
Moreover, \((x,\lambda)\in\Omega\) if and only if
\[
\mathcal L_\rho(x,\mu)
\le\mathcal L_\rho(x,\lambda)
\le\mathcal L_\rho(y,\lambda)
\qquad
\text{for all }(y,\mu)\in
\mathbb{R}^n\times\mathbb{R}^m.
\]
Since \(S\) is nonempty, closed, and convex, its unique minimum-norm
element is $x^*:=P_S(0)
=\operatorname*{argmin}_{x\in S}\Vert x\Vert.$

\subsection{Inertial gradient dynamics with vanishing damping}

The continuous-time viewpoint provides a useful framework for
understanding the acceleration mechanisms behind first-order
optimization methods. A fundamental model is the inertial gradient
system with asymptotically vanishing damping
\[
\ddot{x}(t)+\frac{\alpha}{t}\dot{x}(t)+\nabla f(x(t))=0.
\]
Su, Boyd, and Cand\`es \cite{Su JMLR 2016} showed that, for
\(\alpha=3\), this dynamical system arises as a continuous-time limit
of Nesterov's accelerated gradient method and satisfies $f(x(t))-\min f=O(t^{-2}).$
For \(\alpha>3\), subsequent studies established weak convergence of
the trajectory together with the improved estimate $f(x(t))-\min f=o(t^{-2});$
see \cite{Attouch MP 2018,May Turkish 2017}.

The vanishing-damping principle has also been extended to linearly
constrained convex optimization. For instance, Zeng et al.
\cite{Zeng TAC  2023} studied the Nesterov-type primal--dual dynamical system
\[
\left\{
\begin{aligned}
	\ddot{x}(t)+\frac{\alpha}{t}\dot{x}(t)
	&=
	-\nabla f(x(t))
	-A^\top\bigl(\lambda(t)+\theta t\dot{\lambda}(t)\bigr)
	-A^\top(Ax(t)-b),\\
	\ddot{\lambda}(t)+\frac{\alpha}{t}\dot{\lambda}(t)
	&=
	A\bigl(x(t)+\theta t\dot{x}(t)\bigr)-b.
\end{aligned}
\right.
\]
This and related second-order primal--dual systems yield
\(O(t^{-2})\) estimates for the primal--dual gap, objective residual,
or feasibility violation under suitable parameter conditions; see
\cite{He SICON 2021,Bot JDE 2021,Attouch JOTA 2022,
	He Applied Analysis 2023}.
More recently, He, Huang, Xiao, and Fang
\cite{He arXiv 2026} improved the objective residual and feasibility
violation estimates to \(o(t^{-2})\) and established convergence to a
KKT point for an unregularized primal--dual dynamical system. However, the above primal--dual dynamical systems are second order in both the
primal and dual variables, which makes them more involved from the viewpoint
of numerical computation. Therefore, this computational complexity has motivated the study of mixed-order primal–dual dynamics, in which the inertial term is introduced only into the
primal equation. Along this line, He, Hu, and Fang
\cite{He Automatica 2022} proposed a Nesterov-type mixed-order primal--dual
dynamical system and showed that, even with a first-order dual equation,
the \(O(t^{-2})\)-type convergence rates for the objective residual and
feasibility violation can still be achieved. More recently, Li, Hu, He,
and Xiao \cite{Li JOTA 2025} developed a general mixed-order framework with
time-dependent parameters and derived fast convergence estimates of order
\(O(1/(t^2\beta(t)))\). A natural question is therefore whether one can develop a mixed-order primal–dual dynamical system for linearly constrained convex optimization that achieves \(o(t^{-2})\) convergence rates for both the objective residual and the feasibility violation.

\subsection{Tikhonov regularization and strong convergence}

Fast decay of the objective residual \(f(x(t))-\min f\) does not by
itself determine which solution is selected when the solution set is
not a singleton. A standard device for enforcing a canonical limit is
Tikhonov regularization, which consists of adding a vanishing term
\(\varepsilon(t)x(t)\) to the dynamics. For the inertial dynamical system
\[
\ddot{x}(t)+\frac{\alpha}{t}\dot{x}(t)
+\nabla f(x(t))+\varepsilon(t)x(t)=0,
\]
under suitable conditions on \(\varepsilon\), the regularization term
can enforce strong convergence to the minimum-norm solution. The
analysis initiated in \cite{Attouch JMAA 2018} revealed a basic
rate--selection trade-off: a rapidly vanishing coefficient preserves
accelerated function-value estimates, whereas a more slowly vanishing
coefficient supplies enough regularization to obtain strong
convergence. Related phenomena have subsequently been observed for other
inertial dynamical systems; see \cite{Bot MP 2021,Alecsa SIAM JO 2021,Attouch JDE 2022,
	Attouch AMO 2023,Attouch MMOR 2024}. In particular, the assumptions used to derive fast rates
and those used to establish minimum-norm strong convergence frequently
correspond to different decay regimes of \(\varepsilon(t)\).
More recently, Attouch and L\'aszl\'o
\cite{Attouch MMOR 2024} considered the Tikhonov-regularized inertial
system,
in which the viscous damping is directly coupled with the Tikhonov
coefficient through \(\delta\sqrt{\varepsilon(t)}\).
For the choice
\(\varepsilon(t)=c/t^2\), the damping coefficient becomes
\(\delta\sqrt{c}/t\), corresponding to the critical Nesterov-type
regime. They established accelerated function-value and velocity
estimates in this regime, while the strong convergence of the whole
trajectory to the minimum-norm solution remained open, even in finite
dimensions, without an eventual inside/outside-ball condition.

Tikhonov regularization has also been introduced into mixed-order
primal--dual systems for linearly constrained optimization. A
representative model is
\[
\left\{
\begin{aligned}
	&\ddot{x}(t)+\gamma(t)\dot{x}(t)
	+\eta(t)\bigl(
	\nabla f(x(t))+A^\top\lambda(t)
	+\rho A^\top(Ax(t)-b)+\varepsilon(t)x(t)
	\bigr)=0,\\
	&\dot{\lambda}(t)
	=t\eta(t)
	\bigl[A(x(t)+\theta t\dot{x}(t))-b\bigr],
\end{aligned}
\right.
\]
where \(\gamma(t)\) and \(\eta(t)\) denote damping and time-scaling
coefficients, respectively. Fast decay estimates and minimum-norm
selection properties have been studied for this and related models in
\cite{Zhu Optimization 2026,Zhu JCAM 2025,Sun JOTA 2025,
	Li JOTA 2025,Li CNSNS 2026}. Nevertheless, the parameter regimes ensuring fast
convergence rates and those ensuring minimum-norm strong convergence
are often disjoint. In several analyses, convergence of the whole trajectory is obtained only under the additional assumption that the trajectory eventually remains either inside or outside the ball centered at the origin with radius equal to the norm of the minimum-norm solution. 
Although strong-convergence results without such an eventual ball
condition have recently been obtained for constant damping or damping
of the form \(\alpha/t^q\) with \(0<q<1\) (see \cite{Battahi ASVAO 2025,Csetnek JEE 2026}), they do not
cover the critical \(1/t\) damping regime arising in
\cite{Attouch MMOR 2024} from the choice
\(\varepsilon(t)=c/t^2\).
Consequently, it remains challenging to establish  little-\(o\) convergence
rates and minimum-norm strong convergence under the same parameter
assumptions, without imposing an eventual inside/outside-ball
condition, for Tikhonov-regularized mixed-order primal--dual dynamics.

\subsection{Explicit and implicit Hessian-driven damping}

Although asymptotically vanishing damping is central to acceleration,
its weak dissipation may produce pronounced oscillations, especially
for ill-conditioned objectives. To attenuate these oscillations, Attouch, Peypouquet, and Redont
\cite{Attouch JDE 2016} investigated the Nesterov-type inertial
dynamical system with Hessian-driven damping
\[
\ddot{x}(t)+\frac{\alpha}{t}\dot{x}(t)
+\beta\nabla^2 f(x(t))\dot{x}(t)+\nabla f(x(t))=0.
\]
The term \(\nabla^2 f(x(t))\dot{x}(t)\) provides
curvature-dependent dissipation and helps attenuate oscillations.
For this dynamical system, accelerated function-value estimates and
rapid gradient decay were established. Hessian-driven damping was later incorporated into primal--dual
dynamics for linearly constrained convex optimization; see
\cite{He Optimization 2025,Niederlander SICON 2021,Sun CNSNS 2026}.
Such direct formulations explicitly involve \(\nabla^2 f\), requiring
second-order smoothness and Hessian--vector products. An alternative
is to encode the Hessian effect implicitly by evaluating the gradient
at the extrapolated point \(x(t)+\beta(t)\dot{x}(t)\). When \(f\) is
twice differentiable, the formal first-order expansion $\nabla f\bigl(x(t)+\beta(t)\dot{x}(t)\bigr)
\approx
\nabla f(x(t))
+\beta(t)\nabla^2f(x(t))\dot{x}(t)$
explains this interpretation.The resulting dynamics therefore use only gradient evaluations and
require neither Hessian evaluations nor twice differentiability of
\(f\). This implicit approach has been studied for unconstrained
inertial dynamics in
\cite{Alecsa SIAM JO 2021,Laszlo COAP 2025}.
Building on these developments, Li, He, and Xiao
\cite{Li CNSNS 2026} extended implicit Hessian damping to linearly
constrained convex optimization. The dynamical system developed in
\cite{Li CNSNS 2026} is second order in both the primal and dual
variables, and its dual equation contains several extrapolation and
correction terms. More importantly, the analysis in
\cite{Li CNSNS 2026} requires different parameter regimes for fast
convergence rates and for strong convergence to the minimum-norm
solution. This naturally
raises the question of whether it is possible to construct a
primal--dual dynamical system with implicit Hessian damping and a
simpler dual equation, for which fast convergence rates
and minimum-norm strong convergence can be established under the same
set of assumptions.

The preceding discussion raises the following questions for primal--dual dynamics associated with
linearly constrained convex optimization:
\begin{center}
	\setlength{\fboxsep}{8pt}
	\fbox{
		\begin{minipage}{0.92\textwidth}
			\begin{enumerate}[label=\textup{(Q\arabic*)},
				leftmargin=3.2em,nosep]
				\item Can the convergence rates be improved from
				\(O(\varepsilon(t))\) and \(O(\sqrt{\varepsilon(t)})\) to
				\(o(\varepsilon(t))\) and \(o(\sqrt{\varepsilon(t)})\), respectively,
				while ensuring that these improved rates and the strong convergence of
				the primal trajectory to the minimum-norm solution hold under the same
				parameter assumptions?
				
				\item Can convergence of the whole primal trajectory be proved
				without an eventual inside/outside-ball condition, including for
				the critical Tikhonov regularization coefficient \(\varepsilon(t)=c/t^2\)?
				
				\item Can trajectory oscillations be effectively suppressed without
				requiring either twice differentiability of the objective or Hessian
				evaluations?
			\end{enumerate}
		\end{minipage}
	}
\end{center}

Motivated by these questions, we adopt the
\(\sqrt{\varepsilon(t)}\)-damping mechanism of
\cite{Attouch MMOR 2024}, evaluate both the objective gradient and the
quadratic penalty at \(x(t)+\beta(t)\dot{x}(t)\), and retain a first-order
dual equation. This leads to the following mixed-order primal--dual
dynamical system:
\begin{equation}\label{original-system}
	\left\{
	\begin{aligned}
		\ddot x(t)
		&+\delta\sqrt{\varepsilon(t)}\,\dot x(t)
		+\nabla f\bigl(x(t)+\beta(t)\dot x(t)\bigr)
		+A^\top\lambda(t)\\
		&\qquad
		+\rho A^\top\left(
		A\bigl(x(t)+\beta(t)\dot x(t)\bigr)-b
		\right)
		+\varepsilon(t)x(t)=0,\\[1mm]
		\dot\lambda(t)
		&=
		\frac{\gamma}{\sqrt{\varepsilon(t)}}
		\left[
		A\left(
		x(t)+\frac{\gamma}{\sigma\sqrt{\varepsilon(t)}}\dot x(t)
		\right)-b
		\right].
	\end{aligned}
	\right.
\end{equation}
Here \(\delta,\rho,\sigma,\gamma>0\),
\(\varepsilon\in C^2([t_0,\infty))\) is positive and nonincreasing with
\(\varepsilon(t)\to0\); and
\(\beta:[t_0,\infty)\to[0,\infty)\) is continuous and uniformly bounded:
there exists \(\overline{\beta}\ge0\) such that $0\le\beta(t)\le\overline{\beta}$ for all $t\ge t_0$.
When \(\varepsilon(t)=c/t^2\), the damping coefficient becomes
\(\delta\sqrt c/t\), so \eqref{original-system} includes the critical
Nesterov-type regime.

\subsection{Main contributions}

The main contributions of this paper are summarized as follows.
\vskip 5pt

\begin{enumerate}[label=\textup{(\roman*)},leftmargin=1.8em]	
	\item \emph{\textbf{A unified mixed-order primal--dual framework with
			Tikhonov regularization and implicit Hessian damping.}}
	We extend the unconstrained TRIGS dynamics of Attouch and L\'aszl\'o
	\cite{Attouch MMOR 2024} to linearly constrained convex optimization
	through a mixed-order primal--dual dynamical system incorporating both
	Tikhonov regularization and implicit Hessian damping. Evaluating the
	objective gradient and the quadratic penalty at
	\(x(t)+\beta(t)\dot{x}(t)\) introduces the implicit Hessian effect.
	Unlike explicit Hessian-driven systems
	\cite{Attouch JDE 2016,Bot MP 2021,Niederlander SICON 2021,
		He Optimization 2025,Sun CNSNS 2026}, the proposed dynamical system \eqref{original-system} requires
	neither Hessian evaluations nor twice differentiability of the objective.
	Moreover, compared with the implicit-Hessian primal--dual dynamical
	system in \cite{Li CNSNS 2026}, whose dual equation contains several
	extrapolation and correction terms, \eqref{original-system} retains a
	simpler first-order dual equation and is more amenable to discretization.
	The proposed framework therefore enables the strong convergence to the minimum-norm solution
	while effectively attenuating trajectory oscillations; see
	Figure~\ref{implicit-hessian-figure}.
	
	\item \emph{\textbf{Minimum-norm strong convergence without eventual
			ball conditions, including the critical Nesterov-type regime.}}
Several existing analyses establish only  $\liminf_{t\to\infty}\|x(t)-x^*\|=0,$ and prove convergence of the whole trajectory only under the additional assumption that, for all sufficiently large \(t\), either
	\(\Vert x(t)\Vert<\Vert x^*\Vert\) or
	\(\Vert x(t)\Vert\ge\Vert x^*\Vert\); see
	\cite{Bot MP 2021,Alecsa SIAM JO 2021,Attouch MMOR 2024,
		Zhu JCAM 2025,Li JOTA 2025}.
However, for \eqref{original-system}, we prove
	\(\Vert x(t)-x^*\Vert\to0\) for a general Tikhonov coefficient without
	this eventual inside/outside-ball assumption.
	In particular, when \(\varepsilon(t)=c/t^2\), the damping coefficient
	becomes \(\delta\sqrt{c}/t\), corresponding to the critical
	Nesterov-type regime. In the finite-dimensional unconstrained case,
	our result answers the open question raised by Attouch and L\'aszl\'o
	\cite{Attouch MMOR 2024} of whether the strong convergence of the whole
	trajectory can be established in this critical regime, while the general
	result extends this conclusion to linearly constrained optimization.
	Previous strong-convergence results without the ball condition
	\cite{Battahi ASVAO 2025,Csetnek JEE 2026} apply to constant damping or
	\(\alpha/t^q\) with \(q<1\), but not to the critical damping of order
	\(1/t\).
	
	\item \emph{\textbf{Improved little-\(o\) rates and strong convergence
			under the same parameter conditions.}}
	In many Tikhonov-regularized inertial dynamical systems, fast rates and the
	strong convergence to the minimum-norm solution require different decay regimes of the
	regularization coefficient; see
	\cite{Attouch JMAA 2018,Alecsa SIAM JO 2021,Zhu Optimization 2026,
		Zhu JCAM 2025,Li JOTA 2025}.
	For \eqref{original-system}, the same parameter assumptions guarantee
	both the strong convergence to the minimum-norm solution and the improved convergence rates
	\(o(\varepsilon(t))\) for the Lagrangian gap, objective residual, and
	feasibility violation, together with
	\(o(\sqrt{\varepsilon(t)})\) for the velocity norm.
	In particular, for \(\varepsilon(t)=c/t^2\), these estimates become
	\(o(t^{-2})\) for the Lagrangian gap, objective residual, and feasibility
	violation and \(o(t^{-1})\) for the velocity norm, improving the
	corresponding \(O(t^{-2})\) and \(O(t^{-1})\) decay estimates.
\end{enumerate}

\vskip 5pt

To illustrate the oscillation-attenuation effect of implicit Hessian damping, we consider a family of linearly equality-constrained convex quadratic problems with \(\kappa\in\{100,400\}\):
\[
\min_{x\in\mathbb R^3}
f_\kappa(x):=
\frac12\left(x_1^2+\kappa x_2^2\right)
\qquad\text{subject to}\qquad
x_1+x_2+x_3=0.
\]
The choices \(\kappa=100\) and \(\kappa=400\) represent two
different curvature levels, with the latter producing a more
pronounced high-curvature direction. The dynamical system~\eqref{original-system} was numerically integrated over
\([1,30]\) using the MATLAB solver \texttt{ode15s}, with
\(\varepsilon(t)=t^{-2}\), \(\delta=5\), \(\rho=30\),
\(\gamma=1\), and \(\sigma=5/2\). The initial conditions were
\(x(1)=(1,1,0)^\top\), \(\dot{x}(1)=0\), and \(\lambda(1)=0\).
In both cases, the choices \(\beta=0\) and \(\beta=0.05\) are compared.
\begin{figure}[!h]
	\centering
	\includegraphics[width=0.94\textwidth]
	{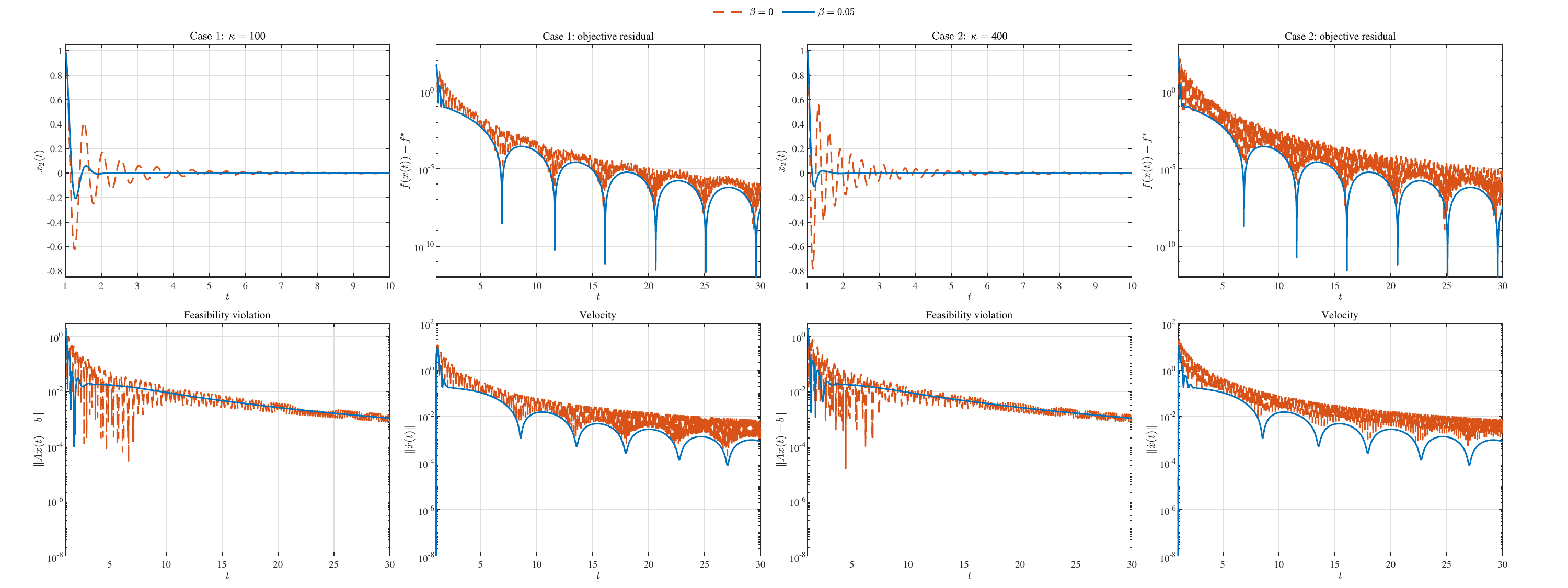}
	\caption{Effect of implicit Hessian damping for two curvature levels.
		The left two columns correspond to Case~1 with \(\kappa=100\), while
		the right two columns correspond to Case~2 with \(\kappa=400\).
		Within each case, the upper-left, upper-right, lower-left, and
		lower-right panels show \(x_2(t)\), the objective residual, the
		feasibility violation, and the velocity, respectively.
		In both cases, the choice \(\beta=0.05\) markedly attenuates the
		oscillations observed for \(\beta=0\) and produces substantially
		smoother decay, illustrating the curvature-adaptive damping effect
		of the implicit Hessian term.}
	\label{implicit-hessian-figure}
\end{figure}

\subsection{Organization}

The rest of the paper is organized as follows. Section~\ref{well-posedness} introduces the
proposed dynamical system, establishes its well-posedness, and proves its
exact equivalence with the time-rescaled system. Section~\ref{fast convergence} develops the Lyapunov analysis and derives fast convergence rates of order \(O(\varepsilon(t))\). Section~\ref{strong convergence} proves strong convergence to the minimum-norm primal solution without an eventual ball condition and improves the preceding estimates to the corresponding little-\(o\) rates. Section~\ref{particular case} specializes the general results to specific choices of the Tikhonov regularization coefficient, with particular emphasis on the critical choice \(\varepsilon(t)=c/t^2\). Finally, Section~\ref{conclusion-section} concludes the paper.

	\section{Well-posedness and exact equivalence under the nonlinear time change}\label{well-posedness}
	In this section, we first establish the existence and uniqueness of a
	global solution to the dynamical system \eqref{original-system}. We then
	introduce a nonlinear change of time and prove the exact equivalence
	between the original system and its time-rescaled formulation, whose
	structure is more suitable for the Lyapunov analysis developed in the
	subsequent sections.
	
\begin{proposition}[Existence and uniqueness of a global solution]
	\label{well-posedness-proposition}
	Let \(t_0>0\). Assume that $f:\mathbb R^n\to\mathbb R$ is convex and continuously differentiable and that \(\nabla f\) is
	globally \(L\)-Lipschitz continuous for some \(L>0\). Assume further
	that $\varepsilon\in C([t_0,+\infty);(0,+\infty))$ and $\beta\in C([t_0,+\infty);[0,+\infty))$. Let $\gamma>0, \delta>0, \rho>0, \sigma>0.$
	Then, for every initial condition $	x(t_0)=x_0, \dot x(t_0)=v_0, \lambda(t_0)=\lambda_0,$ the dynamical system \eqref{original-system} admits a unique global classical solution satisfying $x\in C^2([t_0,+\infty);\mathbb R^n), \lambda\in C^1([t_0,+\infty);\mathbb R^m).$
\end{proposition}

\begin{proof}
	Set $v(t):=\dot x(t)$ and $Y(t):=(x(t),\lambda(t),v(t)).$
	Then \eqref{original-system} is equivalent to
	\[
	\dot Y(t)=\mathcal G(t,Y(t)),
	\]
	where
	\[
	\mathcal G(t,x,\lambda,v)
	:=
	\begin{pmatrix}
		v\\[1mm]
		\displaystyle
		\frac{\gamma}{\sqrt{\varepsilon(t)}}
		\left[
		A\left(
		x+\frac{\gamma}{\sigma\sqrt{\varepsilon(t)}}v
		\right)-b
		\right]\\[3mm]
		\displaystyle
		-\delta\sqrt{\varepsilon(t)}\,v
		-\nabla f\bigl(x+\beta(t)v\bigr)
		-A^\top\lambda
		-\rho A^\top
		\left[
		A\bigl(x+\beta(t)v\bigr)-b
		\right]
		-\varepsilon(t)x
	\end{pmatrix}.
	\]
	Fix an arbitrary finite \(T>t_0\). Since \(\varepsilon(\cdot)\) is continuous
	and strictly positive and \(\beta(\cdot)\) is continuous, the constants
	\[
	\varepsilon_T^-:=
	\min_{t\in[t_0,T]}\varepsilon(t)>0,
	\quad
	\varepsilon_T^+:=
	\max_{t\in[t_0,T]}\varepsilon(t)<+\infty\quad\text{and}\quad	\beta_T:=
	\max_{t\in[t_0,T]}\beta(t)<+\infty
	\]
	are well defined. Thus, by using the global Lipschitz continuity of \(\nabla f\), we obtain that for any $Y_1=(x_1,\lambda_1,v_1), Y_2=(x_2,\lambda_2,v_2)\in\mathbb{R}^n\times\mathbb{R}^m\times\mathbb{R}^n$,
	\[
	\begin{aligned}
		&\left\|
		\nabla f(x_1+\beta(t)v_1)
		-\nabla f(x_2+\beta(t)v_2)
		\right\|\\
		&\qquad\le
		L\|x_1-x_2\|
		+L\beta_T\|v_1-v_2\|.
	\end{aligned}
	\]
	The remaining terms are linear in \(x,\lambda,v\). Consequently,
	for every \(t\in[t_0,T]\),
	\[
	\begin{aligned}
		\|\mathcal G(t,Y_1)-\mathcal G(t,Y_2)\|
		\le&
		\left(
		\frac{\gamma \Vert A\Vert}{\sqrt{\varepsilon_T^-}}
		+L+\rho \Vert A\Vert^2+\varepsilon_T^+
		\right)\|x_1-x_2\|
		+\Vert A\Vert\|\lambda_1-\lambda_2\|\\
		&
		+\left(
		1+\frac{\gamma^2\Vert A\Vert}{\sigma\varepsilon_T^-}
		+\delta\sqrt{\varepsilon_T^+}
		+(L+\rho \Vert A\Vert^2)\beta_T
		\right)\|v_1-v_2\|.
	\end{aligned}
	\]
	Hence there exists \(K_T>0\) such that
	\[
	\|\mathcal G(t,Y_1)-\mathcal G(t,Y_2)\|
	\le K_T\|Y_1-Y_2\|,
	\qquad
	t\in[t_0,T].
	\]
	Thus \(\mathcal G\) is continuous in \(t\) and locally uniformly
	Lipschitz continuous in the state variable. The Cauchy--Lipschitz
	theorem therefore gives a unique maximal solution
	\[
	Y:[t_0,T_{\max})
	\longrightarrow
	\mathbb R^n\times\mathbb R^m\times\mathbb R^n.
	\]
	
	We now prove that \(T_{\max}=+\infty\). Suppose, to the contrary, that
	\(T_{\max}<+\infty\). Applying the preceding estimates on the compact
	interval \([t_0,T_{\max}]\), we obtain
	\[
	\|\mathcal G(t,Y)-\mathcal G(t,0)\|
	\le K_{T_{\max}}\|Y\|.
	\]
	Moreover,
	\[
	\mathcal G(t,0,0,0)
	=
	\begin{pmatrix}
		0\\[1mm]
		\displaystyle
		-\frac{\gamma}{\sqrt{\varepsilon(t)}}b\\[2mm]
		\displaystyle
		-\nabla f(0)+\rho A^\top b
	\end{pmatrix},
	\]
	which is bounded on \([t_0,T_{\max}]\). Therefore, there exists
	\(C>0\) such that
	\[
	\|\mathcal G(t,Y)\|
	\le C(1+\|Y\|),
	\qquad
	t\in[t_0,T_{\max}].
	\]
	Since
	\[
	Y(t)=Y(t_0)+\int_{t_0}^t\mathcal G(s,Y(s))\,ds,
	\]
	Grönwall's inequality yields
	\[
	1+\|Y(t)\|
	\le
	\bigl(1+\|Y(t_0)\|\bigr)e^{C(t-t_0)},
	\qquad
	t<T_{\max}.
	\]
	Thus \(Y(t)\) remains bounded as \(t\uparrow T_{\max}\). It follows that there exists \(R>0\) such that
	\[
	\|Y(t)\|\le R
	\qquad (t<T_{\max}).
	\]
	Since the set $[t_0,T_{\max}] \times \overline B(0,R)$ is compact and \(\mathcal G\) is continuous, it is bounded on this
	set. Hence there exists \(D>0\) such that
	\[
	\|\dot Y(t)\|
	=
	\|\mathcal G(t,Y(t))\|
	\le D
	\qquad(t<T_{\max}).
	\]
 It follows that \(Y(t)\) has a finite limit as
	\(t\uparrow T_{\max}\). Applying the local existence theorem at this
	limit extends the solution beyond \(T_{\max}\), contradicting its
	maximality. Therefore,
	\[
	T_{\max}=+\infty.
	\]
	Finally, \(\mathcal G\) is continuous, so \(Y=(x,\lambda,v)\in C^1\).
	Since $\dot x=v$ and \(v\in C^1\), we conclude that
	\[
	x\in C^2([t_0,+\infty);\mathbb R^n),
	\qquad
	\lambda\in C^1([t_0,+\infty);\mathbb R^m).
	\]
\end{proof}

Following the time-rescaling approach of Attouch, Bo\c{t}, Hulett, and Nguyen
\cite{Attouch SICON 2026}, we introduce the nonlinear change of time $s(t)=\int_{t_0}^{t}\sqrt{\varepsilon(u)}\,du,$
which leads to a more tractable dissipative structure for the subsequent
Lyapunov analysis. The following proposition establishes the exact
equivalence between the original system and its time-rescaled formulation.

	\begin{proposition}[Exact equivalence under the nonlinear time change]
		\label{time-change-lemma}
		 Suppose that the assumptions of Proposition \ref{well-posedness-proposition} hold. And assume further that $\varepsilon\in C^1([t_0,+\infty);(0,+\infty))$. Define
		\begin{equation}\label{time-change}
			s=s(t):=\int_{t_0}^t\sqrt{\varepsilon(u)}\,du,
			\qquad
			s_\infty:=\int_{t_0}^{\infty}\sqrt{\varepsilon(u)}\,du
			\in(0,\infty].
		\end{equation}
		Then \(s:[t_0,\infty)\to[0,s_\infty)\) is strictly increasing.
		Denote its inverse by \(t=t(s)\), and set
		\begin{equation}\label{B-p}
			a(s):=\frac{1}{\varepsilon(t(s))},
			\qquad
			p(s):=\frac{a'(s)}{a(s)}
			=-\frac{\dot\varepsilon(t(s))}
			{\varepsilon(t(s))^{3/2}}.
		\end{equation}
		Consider the transformed dynamical system
		\begin{equation}\label{transformed-system}
			\left\{
			\begin{aligned}
				X''(s)
				&+\left(\delta-\frac{p(s)}{2}\right)X'(s)
				+a(s)\nabla_x\mathcal{L}_{\rho}
				(\widehat X(s),\Lambda(s))
				+X(s)=0,\\
				\Lambda'(s)
				&=
				\gamma a(s)
				\left[
				A\left(X(s)+\frac{\gamma}{\sigma}X'(s)\right)-b
				\right],
			\end{aligned}
			\right.
		\end{equation}
		where
		\begin{equation}\label{implicit-point-s}
			\widehat X(s)
			=
			X(s)+\frac{\beta(t(s))}{\sqrt{a(s)}}X'(s).
		\end{equation}
		Then the following statements hold:
		\begin{enumerate}[label=\textup{(\roman*)},leftmargin=2.5em]
			\item
			If \((x,\lambda)\) is a solution of
			\eqref{original-system}, then
			\[
			X(s):=x(t(s)),
			\qquad
			\Lambda(s):=\lambda(t(s)),
			\qquad
			0\le s<s_\infty,
			\]
			is a solution of \eqref{transformed-system}, with initial conditions
			\begin{equation}\label{initial-data-forward}
				X(0)=x(t_0),
				\qquad
				X'(0)=\frac{\dot x(t_0)}{\sqrt{\varepsilon(t_0)}},
				\qquad
				\Lambda(0)=\lambda(t_0).
			\end{equation}
			
			\item
			Conversely, if \((X,\Lambda)\) is a solution of
			\eqref{transformed-system} on \([0,s_\infty)\), then
			\[
			x(t):=X(s(t)),
			\qquad
			\lambda(t):=\Lambda(s(t)),
			\qquad
			t\ge t_0,
			\]
			is a solution of \eqref{original-system}, with initial conditions
			\begin{equation}\label{initial-data-backward}
				x(t_0)=X(0),
				\qquad
				\dot x(t_0)=\sqrt{\varepsilon(t_0)}\,X'(0),
				\qquad
				\lambda(t_0)=\Lambda(0).
			\end{equation}
		\end{enumerate}
	\end{proposition}
	
	\begin{proof}
		Since \(\varepsilon(t)>0\), the function \(s(\cdot)\) in
		\eqref{time-change} is strictly increasing.
		
		We first prove \textup{(i)}. It follows from \eqref{time-change} that
		\[
		\frac{ds}{dt}
		=
		\sqrt{\varepsilon(t)}
		=
		\frac{1}{\sqrt{a(s)}}.
		\]
		Consequently,
		\begin{equation}\label{xdot-transform}
			\dot x(t)
			=
			\sqrt{\varepsilon(t)}\,X'(s(t))
			=
			\frac{1}{\sqrt{a(s(t))}}X'(s(t)).
		\end{equation}
		Moreover,
		\[
		x(t)+\beta(t)\dot x(t)
		=
		X(s)+\frac{\beta(t(s))}{\sqrt{a(s)}}X'(s).
		\]
		Differentiating once more gives
		\[
		\begin{aligned}
			\ddot x(t)
			&=
			\frac{\dot\varepsilon(t)}
			{2\sqrt{\varepsilon(t)}}X'(s(t))
			+\sqrt{\varepsilon(t)}X''(s(t))\frac{ds}{dt}\\
			&=
			\frac{\dot\varepsilon(t)}
			{2\sqrt{\varepsilon(t)}}X'(s(t))
			+\varepsilon(t)X''(s(t)).
		\end{aligned}
		\]
		Substituting these relations into the first equation of
		\eqref{original-system} and multiplying both sides by \(a(s)\), we
		obtain the first equation of \eqref{transformed-system}. For the dual equation, since
		\[
		\dot\lambda(t)
		=
		\Lambda'(s)\frac{ds}{dt}
		=
		\Lambda'(s)\sqrt{\varepsilon(t)},
		\]
		it follows from the dual equation of \eqref{original-system} that
		\[
		\Lambda'(s)
		=
		\frac{\gamma}{\varepsilon(t)}
		\left[
		A\left(
		x(t)+\frac{\gamma}
		{\sigma\sqrt{\varepsilon(t)}}\dot x(t)
		\right)-b
		\right].
		\]
		Together with \eqref{xdot-transform}, this yields the second equation
		of \eqref{transformed-system}. The initial conditions in
		\eqref{initial-data-forward} follow from \(s(t_0)=0\).
		
		We finally prove \textup{(ii)}. Let \((X,\Lambda)\) solve
		\eqref{transformed-system} on \([0,s_\infty)\), and define
		\[
		x(t)=X(s(t)),
		\qquad
		\lambda(t)=\Lambda(s(t)).
		\]
		The chain rule gives
		\[
		\dot x(t)
		=
		\sqrt{\varepsilon(t)}\,X'(s(t)),
		\]
		\[
		\begin{aligned}
			\ddot x(t)
			=
			\frac{\dot\varepsilon(t)}
			{2\sqrt{\varepsilon(t)}}X'(s(t))
			+\varepsilon(t)X''(s(t))
			=
			\varepsilon(t)
			\left[
			X''(s(t))-\frac{p(s(t))}{2}X'(s(t))
			\right],
		\end{aligned}
		\]
		and
		\[
		\dot\lambda(t)
		=
		\sqrt{\varepsilon(t)}\,\Lambda'(s(t)).
		\]
		Moreover,
		\[
		\begin{aligned}
			\widehat X(s(t))
			=
			X(s(t))
			+\frac{\beta(t)}{\sqrt{a(s(t))}}X'(s(t))
			=
			x(t)+\beta(t)\dot x(t).
		\end{aligned}
		\]
		Multiplying the first equation of \eqref{transformed-system} by
		\(\varepsilon(t)=1/a(s(t))\) and using the preceding identities gives
		\[
		\ddot x(t)
		+\delta\sqrt{\varepsilon(t)}\dot x(t)
		+\nabla_x\mathcal{L}_{\rho}
		\bigl(x(t)+\beta(t)\dot x(t),\lambda(t)\bigr)
		+\varepsilon(t)x(t)=0,
		\]
		which is the primal equation of \eqref{original-system}. Similarly,
		the second equation of \eqref{transformed-system} gives
		\[
		\begin{aligned}
			\dot\lambda(t)
			&=
			\sqrt{\varepsilon(t)}\,\gamma a(s(t))
			\left[
			A\left(
			X(s(t))+\frac{\gamma}{\sigma}X'(s(t))
			\right)-b
			\right]\\
			&=
			\frac{\gamma}{\sqrt{\varepsilon(t)}}
			\left[
			A\left(
			x(t)+\frac{\gamma}
			{\sigma\sqrt{\varepsilon(t)}}\dot x(t)
			\right)-b
			\right],
		\end{aligned}
		\]
		which is the dual equation of \eqref{original-system}. Because
		\(s(t_0)=0\), the initial data satisfy
		\eqref{initial-data-backward}. This proves the converse direction.
	\end{proof}
\begin{remark}
	\label{transformed-time-remark}
	The time change does not necessarily map
	\([t_0,\infty)\) onto \([0,\infty)\).
	Accordingly, the transformed system is initially considered on
	\([0,s_\infty)\). Conditions ensuring \(s_\infty=\infty\) are
	given in Proposition~\ref{boundedness-proposition}.
\end{remark}
\section{Fast convergence analysis}\label{fast convergence}
In this section, we develop a Lyapunov analysis for the dynamical system \eqref{transformed-system} and establish the boundedness estimates needed for the convergence analysis. Returning to the original dynamical system \eqref{original-system}, we prove that the Lagrangian gap, objective residual, and feasibility violation are of
order \(O(\varepsilon(t))\), while the velocity norm is of order
\(O(\sqrt{\varepsilon(t)})\). We first identify a KKT multiplier associated with the minimum-norm solution. The range property established below will be used later to control the dual trajectory.
\begin{lemma}
	\label{KKT-lemma}
	Let $t_0>0$ and $x^*=P_S(0)$. Assume that $f:\mathbb{R}^n\to\mathbb{R}$ is convex and continuously
	differentiable, with an $L$-Lipschitz continuous gradient and the solution set $S\neq\emptyset$. Then, there exists a multiplier
	$\lambda^\dagger\in\mathbb{R}^m$ such that
	\begin{equation}\label{KKT}
		Ax^*=b,
		\qquad
		\nabla f(x^*)+A^\top\lambda^\dagger=0.
	\end{equation}
	It can be chosen so that
	\begin{equation*}\label{initial-range}
		\lambda(t_0)-\lambda^\dagger\in\operatorname*{Ran} A.
	\end{equation*}
	For this choice, every solution of \eqref{original-system} satisfies
	\begin{equation}\label{range-invariance}
		\lambda(t)-\lambda^\dagger\in\operatorname*{Ran} A
		\qquad(t\ge t_0).
	\end{equation}
\end{lemma}

\begin{proof}
	Let $h\in\operatorname*{Ker} A$.  Since $Ax^*=b$, every point
	$x^*+\theta h$ is feasible.  The convex function $\theta\longmapsto f(x^*+\theta h)$ has a minimum at $\theta=0$.  Thus,
	\[
	0=\langle \nabla f(x^*),h\rangle
	\qquad(h\in\operatorname*{Ker} A).
	\]
	Thus
	\[
	\nabla f(x^*)\in(\operatorname*{Ker} A)^\perp=\operatorname*{Ran} A^\top.
	\]
	There is consequently a vector $\widetilde\lambda^\dagger\in\mathbb{R}^m$ such that
	\[
	\nabla f(x^*)+A^\top\widetilde\lambda^\dagger=0.
	\]
	Since the feasible set is nonempty, $b\in\operatorname*{Ran} A$.  The right-hand side of
	the dual equation in \eqref{original-system} is a linear combination of
	$Ax(t)$, $A\dot x(t)$, and $b$, and therefore belongs to $\operatorname*{Ran} A$.  Hence
	\[
	\frac{d}{dt}P_{\operatorname*{Ker} A^\top}\lambda(t)=0,
	\]
	so that
	\[
	P_{\operatorname*{Ker} A^\top}\lambda(t)
	=
	P_{\operatorname*{Ker} A^\top}\lambda(t_0)
	\qquad(t\ge t_0).
	\]
	Define
	\[
	\lambda^\dagger
	:=
	P_{\operatorname*{Ran} A}\widetilde\lambda^\dagger
	+P_{\operatorname*{Ker} A^\top}\lambda(t_0).
	\]
	Since $A^\top P_{\operatorname*{Ker} A^\top}=0$ and
	$A^\top P_{\operatorname*{Ran} A}\widetilde\lambda^\dagger
	=A^\top\widetilde\lambda^\dagger$, this adjusted multiplier still satisfies
	\eqref{KKT}.  Moreover,
	\[
	\begin{aligned}
		P_{\operatorname*{Ker} A^\top}
		\bigl(\lambda(t)-\lambda^\dagger\bigr)
		=
		P_{\operatorname*{Ker} A^\top}\lambda(t)
		-P_{\operatorname*{Ker} A^\top}\lambda(t_0)=0.
	\end{aligned}
	\]
	Since $(\operatorname*{Ker} A^\top)^\perp=\operatorname*{Ran} A$ in finite dimensions,
	\eqref{range-invariance} follows.
\end{proof}

The next lemma provides the estimates for the implicit gradient terms
needed in the energy calculations of Lemmas~\ref{main-energy-lemma} and \ref{cross-lemma}.
\begin{lemma}
	\label{implicit-estimates-lemma}
	Let $(X(s),\Lambda(s))$ be a solution of the dynamical system \eqref{transformed-system}, $x^*=P_S(0)$ and $\lambda^\dagger$ constructed in Lemma \ref{KKT-lemma}, and let \(s(t)\), \(s_\infty\), \(t=t(s)\), and \(a(s)\) be defined by
	\eqref{time-change}--\eqref{B-p}. Suppose that the assumptions of
	Proposition \ref{time-change-lemma} hold. Suppose further that
	\(\beta:[t_0,\infty)\to[0,\infty)\) is bounded, that is, there exists a constant $\overline{\beta}\geq 0$
	such that $0\leq\beta(t)\leq\overline{\beta}$ for all $t\geq t_0$. Consequently, for all \(s\in[0,s_\infty)\), $0\leq\beta(t(s))\leq\overline{\beta}$.
	Then, for every \(s\in[0,s_\infty)\), 
		\begin{equation}\label{monotone-implicit}
		\langle\nabla f(X(s))-\nabla f(\widehat X(s)), X'(s)\rangle
		\le0,
	\end{equation}
	and, for every fixed $\mu\in\mathbb{R}^m$,
	\begin{equation}\label{implicit-position-estimate}
		\begin{aligned}
			\langle\nabla f(\widehat X(s))+A^\top\mu,
			X(s)-x^*\rangle
			\ge{}
			f(X(s))-f(x^*)+\langle\mu, AX(s)-b\rangle
			-\frac{L}{2}\Vert\widehat X(s)-X(s)\Vert^2.
		\end{aligned}
	\end{equation}
	Furthermore, 
	\begin{equation}\label{implicit-gradient-norm}
		\begin{aligned}
			\Vert\nabla f(\widehat X(s))+A^\top\lambda^\dagger\Vert^2
			\le
			4L (\mathcal{L}(X(s), \lambda^\dagger)-\mathcal{L}(x^*, \lambda^\dagger))
			+\frac{2L^2\overline\beta^2}{a(s)}\Vert X'(s)\Vert^2,
		\end{aligned}
	\end{equation}
	where $\widehat{X}(s)$ is defined by \eqref{implicit-point-s}.
\end{lemma}

\begin{proof}
	Since $f$ is convex,  monotonicity of $\nabla f$ gives
	\[
\langle\nabla f(\widehat X(s))-\nabla f(X(s)), \widehat X(s)-X(s)\rangle
	\ge0.
	\]
	If $\beta(t(s))>0$, then
	\[
	X'(s)
	=
	\frac{\sqrt{a(s)}}{\beta(t(s))}(\widehat X(s)-X(s)).
	\]
	By the definition of $a(s)$ and the positivity of $\varepsilon$, we have $a(s)>0$.
	Hence
	\[
	\langle\nabla f(\widehat X(s))-\nabla f(X(s)), X'(s)\rangle\ge0.
	\]
  If
	$\beta(t(s))=0$, then $\widehat X(s)=X(s)$ and the left-hand side of
	\eqref{monotone-implicit} is zero. Thus, \eqref{monotone-implicit} follows.
	
	Next, fix $\mu\in\mathbb{R}^m$ and define
	\[
	h_\mu(u):=f(u)+\langle\mu, Au-b\rangle.
	\]
	The gradient of $h_\mu$ is $L$-Lipschitz.  The descent lemma gives
	\begin{equation}\label{descent}
	    h_\mu(X(s))
	    \le
	    h_\mu(\widehat X(s))
	    +\langle\nabla h_\mu(\widehat X(s)), X(s)-\widehat X(s)\rangle
	    +\frac L2\Vert X(s)-\widehat X(s)\Vert^2.
	\end{equation}
	Moreover, it follows from the convexity of $h_{\mu}$ that
	\[
	h_\mu(x^*)
	\ge
	h_\mu(\widehat X(s))
	+\langle\nabla h_\mu(\widehat X(s)), x^*-\widehat X(s)\rangle.
	\]
	Subtracting the above inequality from \eqref{descent} yields
	\[
	\begin{aligned}
		h_\mu(X(s))-h_\mu(x^*)
		\le{}&
		\langle\nabla h_\mu(\widehat X(s)), X(s)-x^*\rangle
		+\frac L2\Vert X(s)-\widehat X(s)\Vert^2.
	\end{aligned}
	\]
	Since $Ax^*=b$, this is precisely
	\eqref{implicit-position-estimate}.
	
	For the last estimate, apply the descent lemma to $h_{\lambda^\dagger}$ at $X(s)-\frac1L\nabla h_{\lambda^\dagger}(X(s))$.
	It gives
	\[
	\begin{aligned}
		h_{\lambda^\dagger}\left(X(s)-\frac1L\nabla h_{\lambda^\dagger}(X(s))\right)
		\le
		h_{\lambda^\dagger}(X(s))-\frac1{2L}\Vert \nabla h_{\lambda^\dagger}(X(s))\Vert^2.
	\end{aligned}
	\]
	Since $x^*$ minimizes $h_{\lambda^\dagger}$,
	\[
	h_{\lambda^\dagger}(x^*)
	\le
	h_{\lambda^\dagger}\left(X(s)-\frac1L\nabla h_{\lambda^\dagger}(X(s))\right).
	\]
	Combining the last two inequalities yields
	\[
	\Vert\nabla h_{\lambda^\dagger}(X(s))\Vert^2
	\le
	2L\bigl(h_{\lambda^\dagger}(X(s))-h_{\lambda^\dagger}(x^*)\bigr)
	=
	2L(\mathcal{L}(X(s),\lambda^\dagger)-\mathcal{L}(x^*,\lambda^\dagger)),
	\]
	i.e.
	\begin{equation}\label{gap-gradient}
		\Vert \nabla f(X(s))+A^\top\lambda^\dagger\Vert^2\leq2L(\mathcal{L}(X(s),\lambda^\dagger)-\mathcal{L}(x^*,\lambda^\dagger)).
	\end{equation}
	Thus,	using $\Vert u+v\Vert^2\le2\Vert u\Vert^2+2\Vert v\Vert^2$,
	\eqref{gap-gradient}, and Lipschitz continuity,
	\[
	\begin{aligned}
		\Vert\nabla f(\widehat X(s))+A^\top\lambda^\dagger\Vert^2
		&\le
		2\Vert\nabla f(X(s))+A^\top\lambda^\dagger\Vert^2
		+2\Vert\nabla f(\widehat X(s))-\nabla f(X(s))\Vert^2\\
		&\le
		4L(\mathcal{L}(X(s),\lambda^\dagger)-\mathcal{L}(x^*,\lambda^\dagger))+2L^2\Vert\widehat X(s)-X(s)\Vert^2.
	\end{aligned}
	\]
	Finally,  by using \eqref{implicit-point-s} and
	$\beta(t(s))\le\overline\beta$, we obtain
	\eqref{implicit-gradient-norm}.
\end{proof}

Using these preliminary estimates, we introduce the main energy function
and derive its fundamental dissipation inequality.
	\begin{lemma}
		\label{main-energy-lemma}

		Let $(X(s),\Lambda(s))$ be a solution of the dynamical system \eqref{transformed-system}, $x^*=P_S(0)$ and $\lambda^\dagger$ constructed in Lemma \ref{KKT-lemma}, and let \(s(t)\), \(s_\infty\), \(t=t(s)\), and \(a(s)\) be defined by
		\eqref{time-change}--\eqref{B-p}. Suppose that the assumptions of Lemma \ref{implicit-estimates-lemma} hold. Assume further that $\varepsilon\in C^2([t_0,\infty);(0,\infty))$ is nonincreasing with $\epsilon\rightarrow 0$ and there exists a constant $s_T\ge 0$ such that
		\begin{equation}\label{energy-positive-condition}
			p(s)< 2(\delta-\frac{\sigma}{\gamma})
			\qquad(s\ge s_T).
		\end{equation}
		Define
		\begin{equation}\label{main-energy}
			\begin{aligned}
				\mathcal E(s):={}&
				a(s)(\mathcal{L}_{\rho}(X(s),\lambda^\dagger)-\mathcal{L}_{\rho}(x^*,\lambda^\dagger))
				+\frac{1}{2}
				\Vert X'(s)+\frac{\sigma}{\gamma}(X(s)-x^*)\Vert^2\\
				&+\frac{\sigma}{2\gamma}
				\left(
				\delta-\frac{\sigma}{\gamma}-\frac{p(s)}2
				\right)
				\Vert X(s)-x^*\Vert^2
				+\frac{\sigma}{2\gamma^2}
				\Vert \Lambda(s)-\lambda^\dagger\Vert^2
				+\frac{1}{2}\Vert X(s)\Vert^2.
			\end{aligned}
		\end{equation}
		Then $\mathcal E(s)\ge0$  for all  $s\in[s_T,s_\infty)$, and
		\begin{equation}\label{main-energy-derivative}
				\begin{aligned}
					\mathcal E'(s)
					\le{}&
					-\left(\frac{\sigma}{\gamma}-p(s)\right)
					a(s)(\mathcal{L}(X(s),\lambda^\dagger)-\mathcal{L}(x^*,\lambda^\dagger))
					-
					\left[
					\delta-\frac{\sigma}{\gamma}-\frac{p(s)}2
					-\frac{\sigma L}{2\gamma}\beta(t(s))^2
					\right]\Vert X'(s)\Vert^2\\
					&-\frac{\rho\beta(t(s))\sqrt{a(s)}}{2}
					\Vert AX'(s)\Vert^2
					-\rho a(s)\left[
					\frac{\sigma}{\gamma}-\frac{p(s)}2
					-\frac{\sigma^2\beta(t(s))}{2\gamma^2\sqrt{a(s)}}
					\right]\Vert AX(s)-b\Vert^2\\
					&-\frac{\sigma}{\gamma}\langle X(s), X(s)-x^*\rangle
					-\frac{\sigma p'(s)}{4\gamma}
					\Vert X(s)-x^*\Vert^2,\qquad \forall s\in [0, s_{\infty}).
			\end{aligned}
		\end{equation}
	\end{lemma}

\begin{proof}
	It follows from the definition of $a(s)$, the positivity of \(\varepsilon\)  and \eqref{energy-positive-condition} that \(\mathcal E(s)\geq0\) for every	\(s\in[s_T,s_\infty)\). Next,
	We differentiate \eqref{main-energy} with respect to $s$. To this end, we compute the derivative of each term in \eqref{main-energy} separately. Since $a'(s)=p(s)a(s)$ 
	we obtain
	\begin{equation*}\label{first-line-derivative}
		\begin{aligned}
			&\frac{d}{ds}
			\left\{
			a(s)(\mathcal{L}_{\rho}(X(s),\lambda^\dagger)-\mathcal{L}_{\rho}(x^*,\lambda^\dagger))
			\right\}\\
			=&
			p(s)a(s)(\mathcal{L}_{\rho}(X(s),\lambda^\dagger)-\mathcal{L}_{\rho}(x^*,\lambda^\dagger))\\
			&+a(s)\langle\nabla f(X(s))+A^\top\lambda^\dagger, X'(s)\rangle
			+\rho a(s)\langle AX(s)-b, AX'(s)\rangle.
		\end{aligned}
	\end{equation*}
	Next, it follows from the primal equation in \eqref{transformed-system} that,
	\[
	\begin{aligned}
		\frac{d}{ds}
		\left(X'(s)+\frac{\sigma}{\gamma}(X(s)-x^*)\right)
		={}&
		-\left(\delta-\frac{\sigma}{\gamma}-\frac {p(s)}{2}\right)X'(s)
		-a(s)\bigl(
		\nabla f(\widehat X(s))+A^\top\lambda^\dagger
		\bigr)\\
		&-a(s)A^\top(\Lambda(s)-\lambda^\dagger)
		-\rho a(s)A^\top(AX(s)-b)\\
		&-\rho\beta(t(s))\sqrt {a(s)}\,A^\top AX'(s)-X(s).
	\end{aligned}
	\]
	Consequently,
	\begin{equation*}\label{kinetic-derivative}
		\begin{aligned}
			&\frac12\frac{d}{ds}
			\Vert X'(s)+\frac{\sigma}{\gamma}(X(s)-x^*)\Vert^2\\
			={}&
			-\left(\delta-\frac{\sigma}{\gamma}-\frac{p(s)}{2}\right)\Vert X'(s)\Vert^2
			-\frac{\sigma}{\gamma}\left(\delta-\frac{\sigma}{\gamma}-\frac{p(s)}{2}\right)
			\langle X(s)-x^*, X'(s)\rangle\\
			&-a(s)\langle X'(s)+\frac{\sigma}{\gamma}(X(s)-x^*),
			\nabla f(\widehat X(s))+A^\top\lambda^\dagger\rangle-\frac{\rho\sigma}{\gamma} a(s)\Vert AX(s)-b\Vert^2\\
			&-a(s)\langle\Lambda(s)-\lambda^\dagger, AX'(s)+\frac{\sigma}{\gamma}(AX(s)-b)\rangle
			-\rho a(s)\langle AX'(s),AX(s)-b\rangle\\
			&-\rho\beta(t(s))\sqrt{a(s)}
			\langle AX'(s), AX'(s)+\frac{\sigma}{\gamma}(AX(s)-b)\rangle
			-\langle X'(s)+\frac{\sigma}{\gamma}(X(s)-x^*), X(s)\rangle.
		\end{aligned}
	\end{equation*}
		The derivative of the third term in \eqref{main-energy} is
	\begin{equation*}\label{position-derivative}
		\begin{aligned}
			&\frac{d}{ds}
			\left[
			\frac{\sigma}{2\gamma}
			\left(\delta-\frac{\sigma}{\gamma}-\frac{p(s)}{2}\right)
			\Vert X(s)-x^*\Vert^2
			\right]\\
			={}&
			\frac{\sigma}{\gamma}
			\left(\delta-\frac{\sigma}{\gamma}-\frac {p(s)}{2}\right)
			\langle X(s)-x^*, X'(s)\rangle
			-\frac{\sigma p'(s)}{4\gamma}\Vert X(s)-x^*\Vert^2.
		\end{aligned}
	\end{equation*}
	The dual equation in \eqref{transformed-system} gives
	\begin{equation*}\label{dual-energy-derivative}
		\begin{aligned}
			\frac{\sigma}{2\gamma^2}\frac{d}{ds}
			\Vert \Lambda(s)-\lambda^\dagger\Vert^2
			&=a(s)\langle\Lambda(s)-\lambda^\dagger, AX'(s)+\frac{\sigma}{\gamma}(AX(s)-b)\rangle.
		\end{aligned}
	\end{equation*}
	Finally,
	\begin{equation*}\label{center-derivative}
		\frac12\frac{d}{ds}\Vert X(s)\Vert ^2=\langle X(s), X'(s)\rangle.
	\end{equation*}
	Thus, summing the above derivative identities, we obtain the derivative of $\mathcal{E}(s)$ with respect to $s$:
	\begin{equation}\label{energy-before-implicit-estimate}
		\begin{aligned}
			\mathcal E'(s)
			={}&
			p(s)a(s)(\mathcal{L}(X(s),\lambda^\dagger)-\mathcal{L}(x^*,\lambda^\dagger))
			+a(s)\langle\nabla f(X(s))-\nabla f(\widehat X(s)),X'(s)\rangle\\
			&-\frac{\sigma}{\gamma}a(s)
			\langle\nabla f(\widehat X(s))+A^\top\lambda^\dagger,X(s)-x^*\rangle
			-\left(\delta-\frac{\sigma}{\gamma}-\frac{p(s)}{2}\right)\Vert X'(s)\Vert^2\\
			&-\rho\left(\frac{\sigma}{\gamma}-\frac {p(s)}{2}\right)
			a(s)\Vert AX(s)-b\Vert^2
			-\frac{\sigma}{\gamma}\langle X(s), X(s)-x^*\rangle
			-\frac{\sigma p'(s)}{4\gamma}\Vert X(s)-x^*\Vert^2\\
			&-\rho\beta(t(s))\sqrt {a(s)}
			\langle AX'(s), AX'(s)+\frac{\sigma}{\gamma}(AX(s)-b)\rangle.
		\end{aligned}
	\end{equation}
	It follows from the definition of $a(s)$, the positivity of \(\varepsilon\) and  \eqref{monotone-implicit} that
	\[
	a(s)\langle\nabla f(X(s))-\nabla f(\widehat X(s)), X'(s)\rangle\le0.
	\]
   Moreover, \eqref{implicit-position-estimate} with $\mu=\lambda^\dagger$ and \eqref{implicit-point-s} yield
	\[
	\begin{aligned}
		&-\frac{\sigma}{\gamma}a(s)
		\langle\nabla f(\widehat X(s))+A^\top\lambda^\dagger, X(s)-x^*\rangle\\
		\le&
		-\frac{\sigma}{\gamma}a(s)(\mathcal{L}(X(s),\lambda^\dagger)-\mathcal{L}(x^*,\lambda^\dagger))
		+\frac{\sigma La(s)}{2\gamma}\Vert \widehat X(s)-X(s)\Vert^2\\
		=&
		-\frac{\sigma}{\gamma}a(s)(\mathcal{L}(X(s),\lambda^\dagger)-\mathcal{L}(x^*,\lambda^\dagger))
		+\frac{\sigma L}{2\gamma}\beta(t(s))^2\Vert X'(s)\Vert^2.
	\end{aligned}
	\]
What's more, a straightforward calculation yields
	\begin{equation*}\label{augmented-square-completion}
		\begin{aligned}
			&-\rho\beta(t(s))\sqrt{a(s)}
			\langle AX'(s), AX'(s)+\frac{\sigma}{\gamma}(AX(s)-b)\rangle\\
			=&
			-\frac{\rho\beta(t(s))\sqrt{a(s)}}{2}\Vert AX'(s)\Vert^2
			-\frac{\rho\beta(t(s))\sqrt {a(s)}}{2}
			\Vert AX'(s)+\frac{\sigma}{\gamma}(AX(s)-b)\Vert^2\\
			&
			+\frac{\rho\sigma^2\beta(t(s))\sqrt {a(s)}}{2\gamma^2}
			\Vert AX(s)-b\Vert^2\\
			\le&
			-\frac{\rho\beta(t(s))\sqrt {a(s)}}{2}\Vert AX'(s)\Vert^2
			+\frac{\rho\sigma^2\beta(t(s))\sqrt {a(s)}}{2\gamma^2}
			\Vert AX(s)-b\Vert^2.
		\end{aligned}
	\end{equation*}
Finally, substituting the above three estimates into \eqref{energy-before-implicit-estimate}, we obtain \eqref{main-energy-derivative}.
\end{proof}

To obtain direct control of the dual variable, we next estimate a
suitable primal--dual cross term.
\begin{lemma}\label{cross-lemma}
	Let $(X(s),\Lambda(s))$ be a solution of the dynamical system \eqref{transformed-system}, and let $x^*=P_S(0)$ and $\lambda^\dagger$ constructed in Lemma \ref{KKT-lemma}. Suppose that the assumptions of Lemma \ref{main-energy-lemma} hold.
	Define
	\begin{equation*}\label{cross-term}
		W(s)
		:=
		\frac1{a(s)}
		\langle \Lambda(s)-\lambda^\dagger,A\left[X'(s)+\frac{\sigma}{\gamma}
		(X(s)-x^*)\right]\rangle.
	\end{equation*}
Assume, in addition, that
\begin{equation*}\label{cross-assumptions}
	0\le p(s)<\frac{\sigma}{\gamma},
	\quad\text{and}\quad
	a(s)\ge\max\left\{1, 16\rho^2\Vert A\Vert^2\right\}.
\end{equation*}
	Then, for every \(s\in[s_T,s_\infty)\), 
	\begin{equation}\label{cross-upper}
			\begin{aligned}
			W'(s)
				\le{}&
				-\frac14
				\Vert A^\top(\Lambda(s)-\lambda^\dagger)\Vert^2
				+8L(\mathcal{L}(X(s),\lambda^\dagger)-\mathcal{L}(x^*,\lambda^\dagger))\\
				&+C_{\mathrm{vel}}\Vert X'(s)\Vert^2
				+C_{\mathrm{pos}}\Vert X(s)-x^*\Vert^2
				+\frac {a(s)}{4}\Vert AX(s)-b\Vert^2
				+\frac4{a(s)^2}\Vert x^*\Vert^2,
		\end{aligned}
	\end{equation}
where
	\begin{align}
	C_{\mathrm{vel}}
	:={}&
	\Vert A\Vert^2\left(\frac{\gamma^2}{\sigma}+\frac{\gamma}{2}\right)
	+4\gamma^2\Vert A\Vert^2+\frac{4\sigma^2}{\gamma^2}
	+2\left(\delta-\frac{\sigma}{\gamma}\right)^2
	+4L^2\overline\beta^2
	+2\rho^2\overline\beta^2\Vert A\Vert^4,\nonumber\\
	C_{\mathrm{pos}}
	:={}&
	\frac{\gamma\Vert A\Vert^2}{2}
	+4\sigma^2\Vert A\Vert^2
	+\frac{4\sigma^4}{\gamma^4}+4.\nonumber
\end{align}
\end{lemma}

\begin{proof}
	Differentiating $W(s)$ with respect to $s$ and using the dynamical system \eqref{transformed-system}, we obtain the following exact expression for its derivative:
	\begin{equation}\label{cross-exact}
		\begin{aligned}
		W'(s)
			={}&
			-p(s)W(s)
			+\gamma\langle AX(s)-b+\frac{\gamma}{\sigma}AX'(s),A[X'(s)+\frac{\sigma}{\gamma}(X(s)-x^*)]\rangle\\
			&-\frac{\delta-\sigma/\gamma-p(s)/2}{a(s)}
			\langle A^\top(\Lambda(s)-\lambda^\dagger),X'(s)\rangle
			-\langle A^\top(\Lambda(s)-\lambda^\dagger),\nabla f(\widehat X(s))+A^\top\lambda^\dagger\rangle\\
			&-\Vert A^\top(\Lambda(s)-\lambda^\dagger)\Vert^2
			-\rho\langle A^\top(\Lambda(s)-\lambda^\dagger),A^\top(AX(s)-b)\rangle\\
			&-\frac{\rho\beta(t(s))}{\sqrt a(s)}
			\langle A^\top(\Lambda(s)-\lambda^\dagger),A^\top AX'(s)\rangle
			-\frac{1}{a(s)}\langle A^\top(\Lambda(s)-\lambda^\dagger),X(s)\rangle.
		\end{aligned}
	\end{equation}
	We next estimate all terms that do not have a definite sign. Throughout the following estimates, we repeatedly use Young's inequality in the form
	\begin{equation}\label{young-one-eighth}
		ab\leq \frac{1}{8}a^2+2b^2.
	\end{equation}
	Since $0\le p(s)<\sigma/\gamma$ and $a(s)\ge1$,
	\[
	\begin{aligned}
		p(s)|W(s)|
		&\overset{\eqref{young-one-eighth}}{\leq}
		\frac18\Vert A^\top(\Lambda(s)-\lambda^\dagger)\Vert^2
		+\frac{2p(s)^2}{a(s)^2}
		\Vert X'(s)+\frac{\sigma}{\gamma}(X(s)-x^*)\Vert^2\\
		&\le
		\frac18\Vert A^\top(\Lambda(s)-\lambda^\dagger)\Vert^2
		+\frac{4\sigma^2}{\gamma^2}\Vert X'(s)\Vert^2
		+\frac{4\sigma^4}{\gamma^4}\Vert X(s)-x^*\Vert^2.
	\end{aligned}
	\]
	Next, it follows from 
	$0<\delta-\sigma/\gamma-{p(s)}/2\le\delta-\sigma/\gamma$ and $a(s)\ge1$ that
	\[
	\begin{aligned}
		\frac{\delta-\sigma/\gamma-{p(s)}/2}{a(s)}
		\left|
		\langle A^\top(\Lambda(s)-\lambda^\dagger),X'(s)\rangle
		\right|\le
		\frac18\Vert A^\top(\Lambda(s)-\lambda^\dagger)\Vert^2
		+2\left(\delta-\frac{\sigma}{\gamma}\right)^2\Vert X'(s)\Vert^2.
	\end{aligned}
	\]
	For the implicit-gradient term, \eqref{young-one-eighth},
	\eqref{implicit-gradient-norm} and $a(s)\ge 1$ give
	\[
	\begin{aligned}
		&\left|
		\langle A^\top(\Lambda(s)-\lambda^\dagger),\nabla f(\widehat X(s))+A^\top\lambda^\dagger\rangle
		\right|\\\
		\le&
		\frac18\Vert A^\top(\Lambda(s)-\lambda^\dagger)\Vert^2
		+2\Vert \nabla f(\widehat X(s))+A^\top\lambda^\dagger\Vert^2\\
		\le&
		\frac18\Vert A^\top(\Lambda(s)-\lambda^\dagger)\Vert^2
		+8L(\mathcal{L}(X(s),\lambda^\dagger)-\mathcal{L}(x^*,\lambda^\dagger))
		+\frac{4L^2\overline\beta^2}{a(s)}\Vert X'(s)\Vert^2\\
		\le&
		\frac18\Vert A^\top(\Lambda(s)-\lambda^\dagger)\Vert^2
		+8L(\mathcal{L}(X(s),\lambda^\dagger)-\mathcal{L}(x^*,\lambda^\dagger))
		+4L^2\overline\beta^2\Vert X'(s)\Vert^2.
	\end{aligned}
	\]
	Moreover, from \eqref{young-one-eighth} and $a(s)\ge16\rho^2\Vert A\Vert^2$, we have
	\[
	\begin{aligned}
		\rho
		\left|
		\langle A^\top(\Lambda(s)-\lambda^\dagger),A^\top(AX(s)-b)\rangle
		\right|&\le
		\frac18\Vert A^\top(\Lambda(s)-\lambda^\dagger)\Vert^2
		+2\rho^2\Vert A\Vert^2\Vert AX(s)-b\Vert^2\\
		&\leq \frac18\Vert A^\top(\Lambda(s)-\lambda^\dagger)\Vert^2+\frac{a(s)}{8}\Vert AX(s)-b\Vert^2.
	\end{aligned}
	\]
	For the final term in \eqref{cross-exact}, \eqref{young-one-eighth} and $a(s)\ge 1$ yield
	\[
	\begin{aligned}
		\frac{1}{a(s)}
		\left|
		\langle A^\top(\Lambda(s)-\lambda^\dagger),X(s)\rangle
		\right|
		&\le
		\frac18\Vert A^\top(\Lambda(s)-\lambda^\dagger)\Vert^2
		+\frac2{a(s)^2}\Vert X(s)\Vert^2\\
		&\le
		\frac18\Vert A^\top(\Lambda(s)-\lambda^\dagger)\Vert^2
		+\frac4{a(s)^2}\Vert X(s)-x^*\Vert^2
		+\frac4{a(s)^2}\Vert x^*\Vert^2\\
		&\le
		\frac18\Vert A^\top(\Lambda(s)-\lambda^\dagger)\Vert^2
		+4\Vert X(s)-x^*\Vert^2
		+\frac4{a(s)^2}\Vert x^*\Vert^2.
	\end{aligned}
	\]
	What's more, it follows from \eqref{young-one-eighth}, $a(s)\ge 1$  and the boundedness of $\beta(t)$ that
	\[
	\begin{aligned}
		&\frac{\rho\beta(t(s))}{\sqrt a(s)}
		\left|
		\langle A^\top(\Lambda(s)-\lambda^\dagger),A^\top AX'(s)\rangle
		\right|\\
		&\qquad\le
		\frac18\Vert A^\top(\Lambda(s)-\lambda^\dagger)\Vert^2
		+\frac{2\rho^2\beta(t(s))^2}{a(s)}
		\Vert A^\top AX'(s)\Vert^2\\
		&\qquad\le
		\frac18\Vert A^\top(\Lambda(s)-\lambda^\dagger)\Vert^2
		+2\rho^2\overline\beta^2\Vert A\Vert^4\Vert X'(s)\Vert^2.
	\end{aligned}
	\]
		It remains to estimate the first scalar product in
	\eqref{cross-exact}.  Applying \eqref{young-one-eighth} once again yields
\begin{equation*}
	\begin{aligned}
		&\gamma\left|\langle AX(s)-b+\frac{\gamma}{\sigma}AX'(s),A[X'(s)+\frac{\sigma}{\gamma}(X(s)-x^*)]\rangle\right|\\
		\leq &\left|
		\gamma\langle AX(s)-b,A[X'(s)+\frac{\sigma}{\gamma}(X(s)-x^*)]\rangle
		\right|+\frac{\gamma^2}{\sigma}
		\left|
		\langle AX'(s),A[X'(s)+\frac{\sigma}{\gamma}(X(s)-x^*)]\rangle
		\right|\\
		\leq &\frac {a(s)}{8}\Vert AX(s)-b\Vert^2
		+\frac{2\gamma^2\Vert A\Vert^2}{a(s)}
		\Vert X'(s)+\frac{\sigma}{\gamma}(X(s)-x^*)\Vert^2\\
		&+\frac{\gamma^2\Vert A\Vert^2}{\sigma}\Vert X'(s)\Vert^2
		+\gamma\Vert A\Vert^2\Vert X'(s)\Vert\Vert X(s)-x^*\Vert\\
		\leq &\frac {a(s)}{8}\Vert AX(s)-b\Vert^2
		+(4\gamma^2+\frac{\gamma^2}{\sigma}+\frac{\gamma}{2})\Vert A\Vert^2\Vert X'(s)\Vert^2
		+(4\sigma^2+\frac{\gamma}{2})\Vert A\Vert^2\Vert X(s)-x^*\Vert^2.
	\end{aligned}
\end{equation*}
	Finally, substituting the above seven estimates into \eqref{cross-exact} yields the conclusion of Lemma \ref{cross-lemma}.
\end{proof}

To establish the boundedness properties needed for the subsequent
analysis of convergence rates and strong convergence, introduce,
for \(\tau>0\), the modified energy
\begin{equation}
	\label{corrected-energy}
	\widetilde{\mathcal E}(s)
	:=
	\mathcal E(s)+\tau W(s).
\end{equation}
We then combine the main energy with the cross term to construct a
corrected energy and establish boundedness of the  trajectory generated by the dynamical system \eqref{transformed-system}.
\begin{proposition}
	\label{boundedness-proposition}
	
	Let \((X,\Lambda)\) be a solution of the dynamical system
	\eqref{transformed-system}, and let \(x^*=P_S(0)\) and
	\(\lambda^\dagger\) be as in Lemma~\ref{KKT-lemma}.
	Suppose that the assumptions of Lemma~\ref{main-energy-lemma} hold. Assume further that there exist \(s_T\in[0,s_\infty)\) and
	\(\eta_1,\eta_2>0\) such that
	\begin{equation}
		\label{bounded-conditions}
		\begin{cases}
			\displaystyle
			\frac{\sigma}{\gamma}-p(s)\geq\eta_1,
			\\[1ex]
			\displaystyle
			\delta-\frac{\sigma}{\gamma}-\frac{p(s)}2
			-\frac{\sigma L}{2\gamma}\beta(t(s))^2
			\geq\eta_2,
			\\[1ex]
			\displaystyle
			\frac{\sigma}{4\gamma}(-p'(s))_+
			\leq
			\frac{\sigma/\gamma+p(s)}8
		\end{cases}
		\qquad s\in[s_T,s_\infty).
	\end{equation}
 Let \(\tau,\omega>0\) satisfy
	\begin{equation}\label{tau-explicit}
		\begin{aligned}
			&0<\tau\le
			\min\left\{
			1,\,
			\frac{\eta_2}{2C_{\mathrm{vel}}},\,
			\frac{\sigma}{4\gamma C_{\mathrm{pos}}},\,
			\frac{\rho\sigma}{2\gamma}
			\right\},\\
			&0<\omega\le\min\biggl\{
			\frac{\eta_1}{6},\,
			\frac{\eta_2}{6},\,
			\frac{\sigma}{6\gamma},\,
			\frac{\tau\kappa_A^2\gamma^2}{6\sigma},
			\frac{\sigma/\gamma}
			{6\left(2\sigma^2/\gamma^2
				+(\sigma/\gamma)(\delta-\sigma/\gamma)+2\right)}
			\biggr\}.
		\end{aligned}
	\end{equation}
	where $C_{\mathrm{vel}}$ and $C_{\mathrm{pos}}$ are the explicit constants
	in Lemma \ref{cross-lemma} and 
	\(
	\kappa_A :=
	\begin{cases}
		\displaystyle
		\min_{\substack{u\in\operatorname{Ran} A\\ \|u\|=1}}
		\|A^\top u\|, & A\neq 0,\\[1mm]
		1, & A=0.
	\end{cases}
	\).
	The threshold \(s_T\) may be increased, if necessary, so that
	\begin{equation}
		\label{bounded-a-condition}
		a(s)\geq
		\max\left\{
		1,\,
		16\rho^2\|A\|^2,\,
		\frac{16\tau L}{\eta_1},\,
		\frac{2\tau\gamma\|A\|}{\sqrt{\sigma}},\,
		\frac{16\sigma^2\bar{\beta}^{\,2}}{\gamma^2}
		\right\},
		\qquad
		s\in[s_T,s_\infty).
	\end{equation}
Then, \(s_\infty=\infty\) and for every \(s\geq s_T\),
	\begin{equation}
		\label{energy-equivalence}
		\frac12\mathcal E(s)
		\leq
		\widetilde{\mathcal E}(s)
		\leq
		\frac32\mathcal E(s),
	\end{equation}
	and there is a  constant $C_0\ge0$ such that
	\begin{equation*}
		\label{corrected-energy-bound}
		\widetilde{\mathcal E}'(s)
		+\omega\widetilde{\mathcal E}(s)
		\leq C_0,
	\end{equation*}
	where $C_0:=
	\left(
	\frac{3\omega}{2}
	+\frac{\sigma}{\gamma}
	+4\tau
	\right)\|x^*\|^2.$
	Consequently,
	\[
	X,\quad X',\quad\Lambda,\quad
	X-x^*,\quad\Lambda-\lambda^\dagger
	\]
	are bounded on \([s_T,\infty)\).
\end{proposition}

\begin{proof}
	Since \(\varepsilon\) is nonincreasing, \(p(s)\geq0\). The first condition in \eqref{bounded-conditions} therefore gives $0\leq p(s)<\frac{\sigma}{\gamma}.$
	Moreover, the second condition yields $\delta-\frac{\sigma}{\gamma}-\frac{p(s)}2>0.$
	Hence, Lemmas~\ref{main-energy-lemma} and
	\ref{cross-lemma} apply. What's more, since \(\dot{\varepsilon}\leq 0\),
	\[
	p(s)
	=
	-\frac{\dot{\varepsilon}(t(s))}
	{\varepsilon(t(s))^{3/2}}
	\geq0.
	\]
	And the first condition in
	\eqref{bounded-conditions} gives
	\[
	0\leq p(s)\leq\frac{\sigma}{\gamma}-\eta_1
	\qquad
	(s\in[s_0,s_\infty)).
	\]
	Moreover, $p(s(t))
	=
	2\frac{d}{dt}
	\left(
	\frac{1}{\sqrt{\varepsilon(t)}}
	\right).$
	Hence \(1/\sqrt{\varepsilon(t)}\) grows at most linearly, which
	implies \(s_\infty=\infty\). Since \(t(s)\to\infty\) and
	\(\varepsilon(t)\to0\), it follows that $a(s)=\frac{1}{\varepsilon(t(s))}\to\infty.$
	Thus, \(s_T\) can be chosen such that
	\eqref{bounded-a-condition} holds.
	
	Next, we prove \eqref{energy-equivalence}.  The definition of $W(s)$ and Young's inequality $ab\le(a^2+b^2)/2$ give
	\begin{align}
		\tau|W(s)|
		\leq &
		\frac{\tau\gamma\Vert A\Vert}{a(s)\sqrt\sigma}
		\left(
		\frac{\sqrt\sigma}{\gamma}
		\Vert \Lambda(s)-\lambda^\dagger\Vert
		\right)
		\Vert X'(s)+\frac{\sigma}{\gamma}(X(s)-x^*)\Vert\nonumber\\
		\leq &\frac{\tau\gamma\Vert A\Vert}{a(s)\sqrt\sigma}
		\left[
		\frac{\sigma}{2\gamma^2}\Vert \Lambda(s)-\lambda^\dagger\Vert^2
		+\frac12
		\Vert X'(s)+\frac{\sigma}{\gamma}(X(s)-x^*)\Vert^2
		\right]\nonumber\\
		\overset{\eqref{bounded-a-condition}}{\leq}&\frac{1}{2}\left[
		\frac{\sigma}{2\gamma^2}\Vert \Lambda(s)-\lambda^\dagger\Vert^2
		+\frac12
		\Vert X'(s)+\frac{\sigma}{\gamma}(X(s)-x^*)\Vert^2
		\right]\label{tauW_estimate}\\
		\overset{\eqref{main-energy}}{\leq}&\frac{1}{2}\mathcal{E}(s)\nonumber.
	\end{align}
	This proves \eqref{energy-equivalence}. Next, it follows from the definition of $\widetilde{\mathcal{E}}(s)$, \eqref{main-energy-derivative} and \eqref{cross-upper} that
	\begin{equation*}
		\begin{aligned}
			\widetilde{\mathcal{E}}'(s)\leq & -[(\frac{\sigma}{\gamma}-p(s))a(s)-8L\tau](\mathcal{L}(X(s),\lambda^\dagger)-\mathcal{L}(x^*,\lambda^\dagger))-\frac{\tau}{4}\Vert A^\top(\Lambda(s)-\lambda^\dagger)\Vert^2\\
			&-(\delta-\frac{\sigma}{\gamma}-\frac{p(s)}{2}-\frac{\sigma L}{2\gamma}\beta^2(t(s))-\tau C_{\mathrm{vel}})\Vert X'(s)\Vert^2-\frac{\rho\beta(t(s))\sqrt{(a(s))}}{2}\Vert AX'(s)\Vert^2\\
			&-a(s)[\rho(\frac{\sigma}{\gamma}-\frac{p(s)}{2}-\frac{\sigma^2\beta(t(s))}{2\gamma^2\sqrt{a(s)}})-\frac{\tau}{4}]\Vert AX(s)-b\Vert^2-(\frac{\sigma p'(s)}{4\gamma}-\tau C_{\mathrm{pos}})\Vert X(s)-x^*\Vert^2\\
			&+\frac{4\tau}{a(s)}\Vert x^*\Vert^2-\frac{\sigma}{\gamma}\langle X(s), X(s)-x^*\rangle.
		\end{aligned}
	\end{equation*}
    For the position term, write
   \[
   -\frac{\sigma}{\gamma}\langle X(s),X(s)-x^*\rangle
   =
   -\frac{\sigma}{\gamma}\Vert X(s)-x^*\Vert^2
   -\frac{\sigma}{\gamma}\langle x^*,X(s)-x^*\rangle.
   \]
   Then,
	\begin{equation*}\label{second_energy_estimate}
	\begin{aligned}
		\widetilde{\mathcal{E}}'(s)\leq & -[(\frac{\sigma}{\gamma}-p(s))a(s)-8L\tau](\mathcal{L}(X(s),\lambda^\dagger)-\mathcal{L}(x^*,\lambda^\dagger))-\frac{\tau}{4}\Vert A^\top(\Lambda(s)-\lambda^\dagger)\Vert^2\\
		&-(\delta-\frac{\sigma}{\gamma}-\frac{p(s)}{2}-\frac{\sigma L}{2\gamma}\beta^2(t(s))-\tau C_{\mathrm{vel}})\Vert X'(s)\Vert^2-\frac{\rho\beta(t(s))\sqrt{(a(s))}}{2}\Vert AX'(s)\Vert^2\\
		&-a(s)[\rho(\frac{\sigma}{\gamma}-\frac{p(s)}{2}-\frac{\sigma^2\beta(t(s))}{2\gamma^2\sqrt{a(s)}})-\frac{\tau}{4}]\Vert AX(s)-b\Vert^2\\
		&-(\frac{\sigma p'(s)}{4\gamma}-\tau C_{\mathrm{pos}}+\frac{\sigma}{\gamma})\Vert X(s)-x^*\Vert^2
		+\frac{4\tau}{a(s)}\Vert x^*\Vert^2  -\frac{\sigma}{\gamma}\langle x^*,X(s)-x^*\rangle.
	\end{aligned}
\end{equation*}
	For the coefficient of the gap term, \eqref{bounded-a-condition} and the first condition in \eqref{bounded-conditions} give
	\[
	-\left[(\frac{\sigma}{\gamma}-p(s))a(s)-8L\tau\right]\leq -\frac{\sigma/\gamma-p(s)}{2}a(s)\leq -\frac{\eta_1}{2}a(s).
	\]
	For the coefficient of the velocity term, \eqref{tau-explicit} and
	the second condition in \eqref{bounded-conditions} give
	\[
	-\left(\delta-\frac{\sigma}{\gamma}-\frac{p(s)}{2}-\frac{\sigma L}{2\gamma}\beta^2(t(s))-\tau C_{\mathrm{vel}}\right)\leq -\frac{1}{2}\left[
	\delta-\frac{\sigma}{\gamma}-\frac p2
	-\frac{\sigma L}{2\gamma}\beta(t(s))^2
	\right]\leq -\frac{\eta_2}{2}.
	\]
	For the coefficient of the feasibility violation term, $p<\sigma/\gamma$, the boundedness of $\beta(t)$, \eqref{bounded-a-condition} and the choice of $\tau$ yield that
	\[
	-a(s)\left[\rho(\frac{\sigma}{\gamma}-\frac{p(s)}{2}-\frac{\sigma^2\beta(t(s))}{2\gamma^2\sqrt{a(s)}})-\frac{\tau}{4}\right]\leq-\frac{\rho\sigma}{4\gamma}a(s).
	\]
	For the coefficient of the position term, the third condition in \eqref{bounded-conditions}, $p(s)<\sigma/\gamma$ and the choice of $\tau$ imply that
	\[
	-(\frac{\sigma p'(s)}{4\gamma}-\tau C_{\mathrm{pos}}+\frac{\sigma}{\gamma})\leq -\frac{\sigma}{2\gamma}.
	\]
	We next estimate the term $
	-\frac{\tau}{4}\left\|A^\top\bigl(\Lambda(s)-\lambda^\dagger\bigr)\right\|^2$.
    If $A\ne0$,  the unit sphere of the
	finite-dimensional space $\operatorname*{Ran} A$ is nonempty and compact.  The continuous function $u\longmapsto\Vert A^\top u\Vert$
	therefore attains its minimum on that sphere.  If this minimum were zero,
	there would exist $u\in\operatorname*{Ran} A$ such that $\Vert u\Vert=1$ and
	$A^\top u=0$.  This would imply
	\[
	u\in\operatorname*{Ran} A\cap\operatorname*{Ker} A^\top=\{0\},
	\]
	which contradicts $\Vert u\Vert=1$.  Hence  $\kappa_A>0$.  Moreover,
	\eqref{range-invariance} gives
	\[
	\Lambda(s)-\lambda^\dagger\in\operatorname*{Ran} A.
	\]
	Applying the definition of $\kappa_A$ to
	$\Lambda(s)-\lambda^\dagger$ gives
	\[
	\Vert A^\top(\Lambda(s)-\lambda^\dagger)\Vert
	\ge
	\kappa_A\Vert \Lambda(s)-\lambda^\dagger\Vert.
	\]
	Consequently,
	\begin{equation}\label{dual_estimate}
		-\frac{\tau}{4}
		\Vert A^\top(\Lambda-\lambda^\dagger)\Vert^2
		\le
		-\frac{\tau\kappa_A^2}{4}
		\Vert \Lambda-\lambda^\dagger\Vert^2.
	\end{equation}
If $A=0$, nonemptiness of the feasible set implies
	$b=0$.  The dual equation in \eqref{transformed-system} then reduces to
	$\Lambda'(s)=0$.  By the construction of the multiplier in Lemma~\ref{KKT-lemma},
	\[
	\lambda^\dagger=\lambda(t_0)=\Lambda(0),
	\]
	and hence
	\[
	\Lambda(s)-\lambda^\dagger=0
	\qquad(0\le s<s_\infty).
	\]
	Since $\kappa_A=1$ in this case, inequality~\eqref{dual_estimate} holds with
	$\kappa_A>0$ for every matrix $A\in\mathbb{R}^{m\times n}$. 
		\begin{equation}\label{W-boundedness-raw}
		\begin{aligned}
			\widetilde{\mathcal{E}}'(s)
			\le{}&
			-\frac{\eta_1}{2}a(s) (\mathcal{L}(X(s),\lambda^\dagger)-\mathcal{L}(x^*,\lambda^\dagger))
			-\frac{\eta_2}{2}\Vert X'(s)\Vert^2
			-\frac{\rho\sigma}{4\gamma}a(s)\Vert AX(s)-b\Vert^2\\
			&-\frac{\tau\kappa_A^2}{4}
			\Vert \Lambda(s)-\lambda^\dagger\Vert^2
			-\frac{\sigma}{2\gamma}\Vert X(s)-x^*\Vert^2
			+\frac{4\tau}{a(s)^2}\Vert x^*\Vert^2-\frac{\sigma}{\gamma}\langle x^*,X(s)-x^*\rangle.
		\end{aligned}
	\end{equation}
	Moreover, Young's inequality gives 
	\begin{equation*}
		-\frac{\sigma}{\gamma}\langle x^*,X(s)-x^*\rangle
		\le
		\frac{\sigma}{4\gamma}\Vert X(s)-x^*\Vert^2
		+\frac{\sigma}{\gamma}\Vert x^*\Vert^2.
	\end{equation*}
Thus, substituting all the preceding estimates into \eqref{W-boundedness-raw}, we obtain
	\begin{equation*}
		\begin{aligned}
		\widetilde{\mathcal{E}}'(s)
			\le{}&
			-\frac{\eta_1}{2}a(s) (\mathcal{L}(X(s),\lambda^\dagger)-\mathcal{L}(x^*,\lambda^\dagger))
			-\frac{\eta_2}{2}\Vert X'(s)\Vert^2
			-\frac{\rho\sigma}{4\gamma}a(s)\Vert AX(s)-b\Vert^2\\
			&-\frac{\tau\kappa_A^2}{4}
			\Vert \Lambda(s)-\lambda^\dagger\Vert^2
			-\frac{\sigma}{4\gamma}\Vert X(s)-x^*\Vert^2
			+\left(\frac{\sigma}{\gamma}+\frac{4\tau}{a(s)^2}\right)
			\Vert x^*\Vert^2.
		\end{aligned}
	\end{equation*}
Adding \(\omega\widetilde{\mathcal E}(s)\) to both sides of the
preceding inequality and using \eqref{energy-equivalence}, together
with \(\|u+v\|^2\leq2\|u\|^2+2\|v\|^2\), yields
\begin{equation*}
	\begin{aligned}
		\widetilde{\mathcal{E}}'(s)+\omega\widetilde{\mathcal{E}}(s)
		\leq&(-\frac{\eta_1}{2}+\frac{3\omega}{2})a(s)(\mathcal{L}(X(s),\lambda^\dagger)-\mathcal{L}(x^*,\lambda^\dagger))+(-\frac{\rho\sigma}{4\gamma}+\frac{3\rho\omega}{4})a(s)\Vert AX(s)-b\Vert^2\\
		&+(\frac{3\omega}{2}-\frac{\eta_2}{2})\Vert X'(s)\Vert^2+(\frac{3\sigma\omega}{4\gamma^2}-\frac{\tau\kappa_A^2}{4})\Vert \Lambda(s)-\lambda^\dagger\Vert^2+(\frac{\sigma}{\gamma}+\frac{4\tau}{{a(s)}^2}+\frac{3\omega}{2})\Vert x^*\Vert^2\\
		&+(\frac{3\omega\sigma}{4\gamma}(\delta-\frac{\sigma}{\gamma}-\frac{p(s)}{2})-\frac{\sigma}{4\gamma}+\frac{3\omega}{2}+\frac{3\omega \sigma^2}{2\gamma^2})\Vert X(s)-x^*\Vert^2,
	\end{aligned}
\end{equation*}
which together with the choice of $\omega$ implies that for all $s\ge s_T$,
\begin{equation*}
	\begin{aligned}
		\widetilde{\mathcal{E}}'(s)+\omega\widetilde{\mathcal{E}}(s)\leq&
		-\frac{\eta_1}{4}a(s)(\mathcal{L}(X(s),\lambda^\dagger)-\mathcal{L}(x^*,\lambda^\dagger))-\frac{\eta_2}{4}\Vert X'(s)\Vert^2-\frac{\rho\sigma}{8\gamma}a(s)\Vert AX(s)-b\Vert^2\\
		&-\frac{\tau\kappa_A^2}{8}\Vert \Lambda(s)-\lambda^\dagger\Vert^2-\frac{\sigma}{8\gamma}\Vert  X(s)-x^*\Vert^2+(\frac{3\omega}{2}+\frac{\sigma}{\gamma}+4\tau)\Vert x^*\Vert^2\\
	\leq&C_0:=(\frac{3\omega}{2}+\frac{\sigma}{\gamma}+4\tau)\Vert x^*\Vert^2.
	\end{aligned}
\end{equation*}
	Multiplication by $e^{\omega s}$ and integration from $s_T$ to $s$ give
	\[
	\widetilde{\mathcal{E}}(s)
	\le
	e^{-\omega(s-s_T)}\widetilde{\mathcal{E}}(s_T)
	+\frac{C_0}{\omega}
	\left(1-e^{-\omega(s-s_T)}\right).
	\]
	Thus $\widetilde{\mathcal{E}}$ is bounded.  Moreover, it follows from \eqref{energy-equivalence} that
	$\mathcal E$ is bounded. Then, the definition of $\mathcal{E}$ implies that 	
	$X(\cdot), \Lambda(\cdot),
	X(\cdot)-x^*$ and $\Lambda(\cdot)-\lambda^\dagger$
	are bounded on $[s_T,\infty)$. Therefore, by using
	\[
	X'(s)
	=
	\left[X'(s)+\frac{\sigma}{\gamma}(X(s)-x^*)\right]
	-\frac{\sigma}{\gamma}(X(s)-x^*),
	\]
	we also obtain that $X'(\cdot)$ is bounded on $[s_T,\infty)$.
\end{proof}
\begin{remark}[One explicit admissible parameter choice]
	\label{explicit-parameter-choice-remark}
	
	Take
	\[
	\varepsilon(t)=\frac{c}{t},\qquad
	\beta(t)\equiv\beta_0,\qquad
	c,t_0,\delta,\rho,\gamma>0,\qquad
	\beta_0\geq0,
	\]
	and set $
	h:=\frac{\delta}{2+L\beta_0^2}$,
	$\sigma:=\gamma h$,
	$\eta_1:=\frac{h}{2}$,
	$\eta_2:=\frac{\delta}{4}.$
	With \(C_{\mathrm{vel}}\) and \(C_{\mathrm{pos}}\) evaluated at
	these parameters, define
	\[
	\begin{aligned}
		\tau
		&:=
		\frac12\min\left\{
		1,\,
		\frac{\eta_2}{2C_{\mathrm{vel}}},\,
		\frac{\sigma}{4\gamma C_{\mathrm{pos}}},\,
		\frac{\rho\sigma}{2\gamma}
		\right\},\\
		\omega
		&:=
		\frac12\min\left\{
		\frac{\eta_1}{6},\,
		\frac{\eta_2}{6},\,
		\frac{\sigma}{6\gamma},\,
		\frac{\tau\kappa_A^2\gamma^2}{6\sigma},\,
		\frac{\sigma/\gamma}
		{6\left(
			2\sigma^2/\gamma^2
			+
			(\sigma/\gamma)(\delta-\sigma/\gamma)
			+
			2
			\right)}
		\right\}.
	\end{aligned}
	\]
	In this case, $
	s(t)=2\sqrt{c}\bigl(\sqrt{t}-\sqrt{t_0}\bigr)$, $p(s(t))=\frac{1}{\sqrt{ct}}$,
	$p'(s(t))=-\frac{1}{2ct}$,
	$a(s(t))=\frac{t}{c}$.
	Define
	\[
	T:=
	\max\left\{
	t_0,\,
	\frac{4}{ch^2},\,
	\frac{4}{c\delta^2},\,
	\frac{1}{c},\,
	c,\,
	16c\rho^2\|A\|^2,\,
	\frac{16c\tau L}{\eta_1},\,
	\frac{2c\tau\gamma\|A\|}{\sqrt{\sigma}},\,
	\frac{16c\sigma^2\beta_0^2}{\gamma^2}
	\right\}
	\]
	and $
	s_T:=2\sqrt{c}\bigl(\sqrt{T}-\sqrt{t_0}\bigr).$
For \(s\geq s_T\), equivalently \(t(s)\geq T\), the identities
above and the definition of  \(T\) ensure
\eqref{bounded-conditions} and \eqref{bounded-a-condition}.
Hence, all the conditions of
Proposition~\ref{boundedness-proposition} are satisfied. The corresponding convergence results for other explicit
	coefficients are given in
	Section~\ref{explicit-coefficients-section}.
\end{remark}

The preceding energy estimates for the dynamical system \eqref{transformed-system} yield the
following fast convergence result for the original system
\eqref{original-system}.
\begin{theorem}
	\label{fast-theorem}
		Let $(x(t),\lambda(t))$ be a trajectory of the dynamical system \eqref{original-system}, $x^*=P_S(0)$ and choose \(\lambda^\dagger\) as in Lemma~\ref{KKT-lemma} so that $
		\lambda(t_0)-\lambda^\dagger\in\operatorname{Ran}A.$ Suppose that the assumptions of Proposition \ref{boundedness-proposition} hold.
	Then
	\begin{align}
	&	\mathcal{L}(x(t),\lambda^\dagger)-\mathcal{L}(x^*,\lambda^\dagger)=O(\varepsilon(t)),\nonumber\quad
	\Vert Ax(t)-b\Vert=O(\varepsilon(t)),\nonumber\\
		&|f(x(t))-f(x^*)|
	=O(\varepsilon(t)),\nonumber\quad
		\Vert \dot x(t)\Vert
		=O(\sqrt{\varepsilon(t)}).
	\nonumber
	\end{align}
\end{theorem}

\begin{proof}
	It follows from the boundedness of $\mathcal{E}$ obtained from Proposition \ref{boundedness-proposition} and  $\mathcal{L}(X(s),\lambda^\dagger)-\mathcal{L}(x^*,\lambda^\dagger) \ge 0$ that there exists $C_{\mathcal E}>0$ such that
	\[
	0\le a(s)(\mathcal{L}(X(s),\lambda^\dagger)-\mathcal{L}(x^*,\lambda^\dagger))\leq a(s)(\mathcal{L}_{\rho}(X(s),\lambda^\dagger)-\mathcal{L}_{\rho}(x^*,\lambda^\dagger)) \le\mathcal E(s)\le C_{\mathcal E}.
	\]
	Thus
	\begin{equation}\label{gap-rate-s}
	\mathcal{L}(X(s),\lambda^\dagger)-\mathcal{L}(x^*,\lambda^\dagger)=O\left(\frac1{a(s)}\right).
	\end{equation}
	To estimate the feasibility residual, set
	\begin{equation}\label{weighted-residual}
		\Phi(s):=
		\frac{\gamma^2}{\sigma}a(s)(AX(s)-b).
	\end{equation}
	Since $a'(s)=p(s)a(s)$,
	\[
	\begin{aligned}
		\Phi'(s)=
		\frac{\gamma^2}{\sigma}p(s)a(s)(AX(s)-b)
		+\frac{\gamma^2}{\sigma}a(s)AX'(s).
	\end{aligned}
	\]
	On the other hand, the dual equation in \eqref{transformed-system} is
	\[
	\Lambda'(s)
	=
	\gamma a(s)(AX(s)-b)
	+\frac{\gamma^2}{\sigma}a(s)AX'(s).
	\]
	Subtracting the preceding two identities and using
	\eqref{weighted-residual} gives
	\begin{equation}\label{weighted-residual-ODE}
		\Phi'(s)
		+\left(\frac{\sigma}{\gamma}-p(s)\right)\Phi(s)
		=
		\Lambda'(s).
	\end{equation}
	Define
	\begin{equation}\label{integrating-factor-K1}
		K(s):=
		\int_{s_T}^{s}
		\left(\frac{\sigma}{\gamma}-p(u)\right)\,du.
	\end{equation}
	The first condition in \eqref{bounded-conditions} implies that $K(\cdot)$ is nondecreasing.
	Multiplying \eqref{weighted-residual-ODE} by $e^{K(s)}$ and integrating
	from $s_T$ to $s$ yields
	\[
	\begin{aligned}
		\Phi(s)
		={}
		e^{-K(s)}\Phi(s_T)
		+\int_{s_T}^{s}
		e^{-(K(s)-K(u))}\Lambda'(u)\,du.
	\end{aligned}
	\]
	Integrating the last term by parts gives the exact identity
	\begin{equation}\label{weighted-residual-representation}
		\begin{aligned}
			\Phi(s)
			={}
			e^{-K(s)}\Phi(s_T)
			+\Lambda(s)-e^{-K(s)}\Lambda(s_T)
			-\int_{s_T}^{s}
			e^{-(K(s)-K(u))}
			\left(\frac{\sigma}{\gamma}-p(u)\right)\Lambda(u)\,du.
		\end{aligned}
	\end{equation}
	Let
	\[
	M_\Lambda:=
	\sup_{u\ge s_T}\Vert \Lambda(u)\Vert<\infty.
	\]
	Since $\sigma/\gamma-p(u)\ge0$,
	\[
	\begin{aligned}
		\int_{s_T}^{s}
		e^{-(K(s)-K(u))}
		\left(\frac{\sigma}{\gamma}-p(u)\right)\,du
		=
		1-e^{-K(s)}
		\le1.
	\end{aligned}
	\]
	It follows from \eqref{weighted-residual-representation} that
	\[
	\Vert \Phi(s)\Vert
	\le
	\Vert \Phi(s_T)\Vert+3M_\Lambda
	\qquad(s\ge s_T).
	\]
	Thus $\Phi$ is bounded. It follows from \eqref{weighted-residual} that
	\begin{equation}\label{residual-rate-s}
		\Vert AX(s)-b\Vert
		=
		O\left(\frac1{a(s)}\right).
	\end{equation}
	Since $a(s(t))=1/\varepsilon(t)$ and $X(s(t))=x(t)$,
	\eqref{gap-rate-s} and \eqref{residual-rate-s} give
	\[
		\mathcal{L}(x(t),\lambda^\dagger)-\mathcal{L}(x^*,\lambda^\dagger)=O(\varepsilon(t)),
	\quad
\text{and}\quad	\Vert Ax(t)-b\Vert=O(\varepsilon(t)).
	\]
	The identity
	\[
	f(x(t))-f(x^*)
	=
	\mathcal{L}(x(t),\lambda^\dagger)-\mathcal{L}(x^*,\lambda^\dagger)
	-\langle \lambda^\dagger,Ax(t)-b\rangle
	\]
	then gives
	\[
	|f(x(t))-f(x^*)|
	\le
	\mathcal{L}(x(t),\lambda^\dagger)-\mathcal{L}(x^*,\lambda^\dagger)
	+\Vert \lambda^\dagger\Vert\Vert Ax(t)-b\Vert
	=
	O(\varepsilon(t)).
	\]
	Finally, the boundedness of
	$X'(\cdot)+\frac{\sigma}{\gamma}(X(\cdot)-x^*)$ and $X(\cdot)-x^*$ implies that
	$X'(\cdot)$ is bounded.  By \eqref{xdot-transform},
	\[
	\dot x(t)
	=
	\sqrt{\varepsilon(t)}\,X'(s(t)),
	\]
	which proves $\Vert \dot{x}(t)\Vert=\mathcal{O}(\sqrt{\varepsilon}(t))$.
\end{proof}

\section{Strong convergence without a ball condition and improved rates}\label{strong convergence}
In this section, without imposing any eventual inside/outside-ball
condition, we prove that the whole primal trajectory converges strongly
to the minimum-norm solution \(x^*=P_S(0)\), while the dual trajectory
converges strongly to the compatible KKT multiplier
\(\lambda^\dagger\). Under the same assumptions, we further improve the
\(O\)-estimates of Theorem~\ref{fast-theorem} to the corresponding
little-\(o\) rates and establish the convergence of the implicit
evaluation point.

We first establish the strong convergence of the primal--dual trajectory
to \((x^*,\lambda^\dagger)\), where \(x^*\) is the minimum-norm primal
solution.
\begin{theorem}
	\label{strong-theorem}
	Let \((x,\lambda)\) be a trajectory of the dynamical system
	\eqref{original-system}. Let \(x^*=P_S(0)\), and choose
	\(\lambda^\dagger\) as in Lemma~\ref{KKT-lemma} so that
	\[
	\lambda(t_0)-\lambda^\dagger\in\operatorname{Ran}A.
	\]
	Suppose that the assumptions of
	Theorem~\ref{fast-theorem} hold. Then
	\[
	x(t)\to x^*
	\qquad\text{and}\qquad
	\lambda(t)\to\lambda^\dagger
	\qquad\text{as }t\to\infty.
	\]
\end{theorem}

\begin{proof}
	It follows from the conclusions of Proposition \ref{boundedness-proposition} and Theorem~\ref{fast-theorem} that
	\begin{equation}\label{gap-residual-zero}
	X(\cdot)\text{ is bounded},\quad	\mathcal{L}(X(s),\lambda^\dagger)-\mathcal{L}(x^*,\lambda^\dagger)\to0,
		\quad\text{and}\quad
		AX(s)-b\to0.
	\end{equation}
	Let $\overline x$ be a cluster point of $X$. Then,  there is a sequence
	$s_k\to\infty$ such that $X(s_k)\to\overline x$.  Passing to the limit in
	\eqref{gap-residual-zero} gives
	\[
	A\overline x=b.
	\]
	Furthermore,
	\[
	\begin{aligned}
		f(\overline x)-f(x^*)
		=
		\lim_{k\to\infty}
		\left[ \mathcal{L}(X(s_k),\lambda^\dagger)-\mathcal{L}(x^*,\lambda^\dagger)
		-\langle \lambda^\dagger,AX(s_k)-b\rangle
		\right]
		=0.
	\end{aligned}
	\]
	Hence every cluster point of $X$ belongs to $S$. Moreover, the projection identity $x^*=P_S(0)$ gives
	\begin{equation}\label{projection-inequality}
		\langle x^*,y-x^*\rangle\ge0
		\qquad(y\in S).
	\end{equation}
	We claim that
	\begin{equation}\label{projection-liminf}
		\liminf_{s\to\infty}
		\langle x^*,X(s)-x^*\rangle\ge0.
	\end{equation}
	If this were false, there would be $\eta>0$ and a sequence $s_k\to\infty$
	such that
	\[
	\langle x^*,X(s_k)-x^*\rangle\le-\eta.
	\]
	The bounded sequence $X(s_k)$ would have a convergent subsequence whose
	limit $\overline x$ belongs to $S$.  Passing to the limit would give
	\[
	\langle x^*,\overline x-x^*\rangle\le-\eta,
	\]
	contradicting \eqref{projection-inequality}. Thus,  this proves
	\eqref{projection-liminf}.  In particular,
	\begin{equation}\label{negative-part}
		\max\left\{
		-\langle x^*,X(s)-x^*\rangle,0
		\right\}\longrightarrow0.
	\end{equation}
	Define
	\begin{equation*}\label{shifted-W}
		\widehat{\mathcal{E}}(s)
		:=
		\widetilde{\mathcal{E}}(s)-\frac12\Vert x^*\Vert^2,
	\end{equation*}
where $\widetilde{\mathcal{E}}(s)$ is constructed by \eqref{corrected-energy}.
	Using
	\[
	\frac12\Vert X(s)\Vert^2-\frac12\Vert x^*\Vert^2
	=
	\frac12\Vert X(s)-x^*\Vert^2
	+\langle x^*,X(s)-x^*\rangle,
	\]
	we obtain the exact decomposition
	\begin{equation}\label{shifted-W-decomposition}
		\begin{aligned}
			\widehat{\mathcal{E}}(s)
			={}&
			a(s)(
		\mathcal{L}_{\rho}(X(s),\lambda^\dagger)-\mathcal{L}_{\rho}(x^*,\lambda^\dagger))
			+\frac12
			\Vert X'(s)+\frac{\sigma}{\gamma}(X(s)-x^*)\Vert^2
			+\frac{\sigma}{2\gamma^2}\Vert \Lambda(s)-\lambda^\dagger\Vert^2\\
			&+\frac{\tau}{a(s)}
			\langle \Lambda(s)-\lambda^\dagger,A[X'(s)+\frac{\sigma}{\gamma}(X(s)-x^*)]\rangle\\
			&+\frac12
			\left[
			1+\frac{\sigma}{\gamma}
			\left(\delta-\frac{\sigma}{\gamma}-\frac {p(s)}2\right)
			\right]\Vert X(s)-x^*\Vert^2
			+\langle x^*,X(s)-x^*\rangle.
		\end{aligned}
	\end{equation}
	 Since, by using \eqref{tauW_estimate}, we obtain
	\begin{equation}\label{block-coercivity}
		\begin{aligned}
			&\frac14
			\Vert X'(s)+\frac{\sigma}{\gamma}(X(s)-x^*)\Vert^2
			+\frac{\sigma}{4\gamma^2}\Vert \Lambda(s)-\lambda^\dagger\Vert^2\\
			&\le
			\frac12\Vert X'(s)+\frac{\sigma}{\gamma}(X(s)-x^*)\Vert^2
			+\frac{\sigma}{2\gamma^2}\Vert \Lambda(s)-\lambda^\dagger\Vert^2
			+\frac{\tau}{a(s)}
			\langle \Lambda(s)-\lambda^\dagger,A[X'(s)+\frac{\sigma}{\gamma}(X(s)-x^*)]\rangle\\
			&
			\leq \frac{3}{2}\Vert X'(s)\Vert^2+\frac{3\sigma^2}{2\gamma^2}\Vert X(s)-x^*\Vert^2+\frac{3\sigma}{4\gamma^2}\Vert\Lambda(s)-\lambda^\dagger\Vert^2.
		\end{aligned}
	\end{equation}
	Thus, for all $\vartheta>0$, we have
	\begin{equation*}
		\begin{aligned}
			&\widehat{\mathcal{E}}'(s)+\vartheta\widehat{\mathcal{E}}(s)\\
			\overset{\eqref{W-boundedness-raw},\eqref{block-coercivity}}{\leq}&-(\frac{\eta_1}{2}-\vartheta)a(s)(	\mathcal{L}(X(s),\lambda^\dagger)-\mathcal{L}(x^*,\lambda^\dagger))-(\frac{\eta_2}{2}-\frac{3\vartheta}{2})\Vert X'(s)\Vert^2\\
			&-(\frac{\tau\kappa_A^2}{4}-\frac{3\vartheta \sigma}{4\gamma^2})\Vert \Lambda(s)-\lambda^\dagger\Vert^2-(\frac{\sigma}{\gamma}-\vartheta)\langle x^*, X(s)-x^*\rangle+\frac{4\tau}{a(s)^2}\Vert x^*\Vert^2\\
			&-(\frac{\sigma}{2\gamma}-\frac{\vartheta}{2}(\frac{\sigma}{\gamma}(\delta-\frac{\sigma}{\gamma}-\frac{p(s)}{2})+1)-\frac{3\vartheta\sigma^2}{2\gamma^2})\Vert X(s)-x^*\Vert^2
			-(\frac{\rho\sigma}{4\gamma}-\frac{\vartheta \rho}{2})a(s)\Vert AX(s)-b\Vert^2.
		\end{aligned}
	\end{equation*}
	Now, choosing explicitly
	\begin{equation*}\label{vartheta-choice}
		\begin{aligned}
			0<\vartheta\le\min\biggl\{
			\frac{\eta_1}{4},\,
			\frac{\eta_2}{6},\,
			\frac{\sigma}{4\gamma},\,
			\frac{\tau\kappa_A^2\gamma^2}{6\sigma},
			\frac{\sigma/\gamma}
			{2\left(3\sigma^2/\gamma^2
				+(\sigma/\gamma)(\delta-\sigma/\gamma)+1\right)}
			\biggr\}.
		\end{aligned}
	\end{equation*}
implies that
\begin{equation*}
	\widehat{\mathcal{E}}'(s)+\vartheta\widehat{\mathcal{E}}(s)\leq \frac{4\tau}{a(s)^2}\Vert x^*\Vert^2-(\frac{\sigma}{\gamma}-\vartheta)\langle x^*,X(s)-x^*\rangle,
\end{equation*}
which together with 
	\[
	\begin{aligned}
		-\left(\frac{\sigma}{\gamma}-\vartheta\right)
		\langle x^*,X-x^*\rangle\le
		\left(\frac{\sigma}{\gamma}-\vartheta\right)
		\max\left\{
		-\langle x^*,X-x^*\rangle,0
		\right\}.
	\end{aligned}
	\]
	yields
	\begin{equation}\label{shifted-W-scalar}
		\begin{aligned}
			\widehat{\mathcal E}'(s)
			+\vartheta\widehat{\mathcal{E}}(s)
			\le{}
			\left(\frac{\sigma}{\gamma}-\vartheta\right)
			\max\left\{
			-\langle x^*,X-x^*\rangle,0
			\right\}
			+\frac{4\tau}{a(s)^2}\Vert x^*\Vert^2.
		\end{aligned}
	\end{equation}
It follows from \eqref{negative-part} and $a(s)\to\infty$ that
	the right-hand side of \eqref{shifted-W-scalar} tends to zero.   Then, for any
	$\eta>0$,  there exists $S_\eta\ge s_T$ such that, for every
	$u\ge S_\eta$,
	\[
	\begin{aligned}
		&\left(\frac{\sigma}{\gamma}-\vartheta\right)
		\max\left\{
		-\langle x^*,X(u)-x^*\rangle,0
		\right\}
		+\frac{4\tau}{a(u)^2}\Vert x^*\Vert^2
		\le\eta.
	\end{aligned}
	\]
	Multiplying \eqref{shifted-W-scalar} by $e^{\vartheta s}$ and
	integrating from $S_\eta$ to $s\ge S_\eta$ give
	\[
	\begin{aligned}
		\widehat{\mathcal E}(s)
		\le{}&
		e^{-\vartheta(s-S_\eta)}
		\widehat{\mathcal E}(S_\eta)\\
		&+\int_{S_\eta}^{s}e^{-\vartheta(s-u)}
		\left[
		\left(\frac{\sigma}{\gamma}-\vartheta\right)
		\max\left\{
		-\langle x^*,X(u)-x^*\rangle,0
		\right\}
		+\frac{4\tau}{a(u)^2}\Vert x^*\Vert^2
		\right]du\\
		\le{}&
		e^{-\vartheta(s-S_\eta)}
		\widehat{\mathcal E}(S_\eta)
		+\eta\int_{S_\eta}^{s}e^{-\vartheta(s-u)}\,du\\
		={}&
		e^{-\vartheta(s-S_\eta)}
		\widehat{\mathcal E}(S_\eta)
		+\frac{\eta}{\vartheta}
		\left(1-e^{-\vartheta(s-S_\eta)}\right).
	\end{aligned}
	\]
	Since $\vartheta>0$, letting $s\to\infty$ yields
	\[
	\limsup_{s\to\infty}\widehat{\mathcal E}(s)
	\le\frac{\eta}{\vartheta}.
	\]
	Moreover, since $\eta>0$ is arbitrary, it follows that
	\begin{equation}\label{W-limsup}
		\limsup_{s\to\infty}\widehat{\mathcal E}(s)\le0.
	\end{equation}
	In addition, \eqref{block-coercivity} and $\delta-\sigma/\gamma-p(s)/2>0$ imply that
	\[
	\widehat{\mathcal{E}}(s)
	\ge
	\frac12\Vert X(s)-x^*\Vert^2
	+\langle x^*,X(s)-x^*\rangle.
	\]
	Combining this inequality with \eqref{projection-liminf} and
	\eqref{W-limsup},
	\[
	\begin{aligned}
		\frac12\limsup_{s\to\infty}\Vert X(s)-x^*\Vert^2
		\le
		\limsup_{s\to\infty}\widehat{\mathcal{E}}(s)
		-\liminf_{s\to\infty}
		\langle x^*,X(s)-x^*\rangle
		\le0.
	\end{aligned}
	\]
	Thus
	\[
	X(s)\to x^*,
	\]
	which yields that
	\[
	\langle x^*,X(s)-x^*\rangle\to0.
	\]
	The decomposition \eqref{shifted-W-decomposition} then implies
	\[
	\liminf_{s\to\infty}\widehat{\mathcal{E}}(s)\ge0.
	\]
	Together with \eqref{W-limsup}, this yields
	\[
	\widehat{\mathcal{E}}(s)\to0.
	\]
	Moreover, \eqref{shifted-W-decomposition} and
	\eqref{block-coercivity}, together with the nonnegativity of the
	remaining terms, yield
	\[
	\frac{\sigma}{4\gamma^2}
	\Vert\Lambda(s)-\lambda^\dagger\Vert^2
	\leq
	\widehat{\mathcal E}(s)
	-\langle x^*,X(s)-x^*\rangle.
	\]
	Since $\widehat{\mathcal E}(s)\to0$ and $\langle x^*,X(s)-x^*\rangle\to0$
	it follows that
	\begin{equation}\label{dual-zero}
		\Lambda(s)-\lambda^\dagger\longrightarrow0.
	\end{equation}
	Finally, since \(s(t)\to\infty\), returning to the original time
	variable gives the conclusion.
\end{proof}
\begin{remark}
		\label{no-ball-condition-remark}
		For several Tikhonov-regularized inertial systems and their primal--dual
		extensions, the analysis without an additional geometric assumption yields
		only $\liminf_{t\to\infty}\Vert x(t)-x^*\Vert=0;$
		see
		\cite{Bot MP 2021,Alecsa SIAM JO 2021,Attouch MMOR 2024,
			Zhu JCAM 2025,Li JOTA 2025}.
		This conclusion guarantees at most the existence of a sequence
		\(t_k\to\infty\) such that \(x(t_k)\to x^*\), but does not ensure convergence
		of the whole trajectory.  In the cited works, the latter is generally
		obtained by additionally assuming that, for all sufficiently large \(t\),
		the trajectory remains either inside or outside the ball
		\(B(0,\Vert x^*\Vert)\), namely, $	\Vert x(t)\Vert<\Vert x^*\Vert\text{ or }
		\Vert x(t)\Vert\ge\Vert x^*\Vert.$
		Theorem~\ref{strong-theorem} removes this eventual inside/outside-ball
		condition: under the same parameter assumptions as those used to establish
		the fast convergence rates, the whole trajectory generated by
		\eqref{original-system} satisfies $	\lim_{t\to\infty}\Vert x(t)-x^*\Vert=0.$
		Thus, the strong convergence result requires no prior information on the
		eventual position of the trajectory relative to
		\(B(0,\Vert x^*\Vert)\).
\end{remark}

The strong-convergence result allows the estimates of
Theorem~\ref{fast-theorem} to be sharpened to little-\(o\) rates, as
stated in the following theorem.
\begin{theorem}[Improved convergence rates]
	\label{little-o-theorem}
	Under the assumptions of Theorem~\ref{strong-theorem}, the estimates of
	Theorem~\ref{fast-theorem} improve to
	\begin{equation*}
		\begin{aligned}
			&\Vert\dot{x}(t)\Vert=o\bigl(\sqrt{\varepsilon(t)}\bigr),
			\qquad
			\mathcal{L}(x(t),\lambda^\dagger)
			-\mathcal{L}(x^*,\lambda^\dagger)
			=o\bigl(\varepsilon(t)\bigr),\nonumber\\
			&\Vert Ax(t)-b\Vert=o\bigl(\varepsilon(t)\bigr),
			\qquad
			\bigl|f(x(t))-f(x^*)\bigr|
			=o\bigl(\varepsilon(t)\bigr).
		\end{aligned}
	\end{equation*}
\end{theorem}

\begin{proof}
	The proof of Theorem~\ref{strong-theorem} established
	\[
	\widehat{\mathcal{E}}(s)\to0
	\quad\text{and}\quad
	\langle x^*,X(s)-x^*\rangle\to0,
	\]
	which together with \eqref{block-coercivity} and $\delta-\sigma/\gamma-p(s)/2>0$ implies that
	\begin{align}
		a(s)	(\mathcal{L}(X(s),\lambda^\dagger)-\mathcal{L}(x^*,\lambda^\dagger))&\to0,\label{BG-zero}\\
		X'(s)+\frac{\sigma}{\gamma}(X(s)-x^*)
		&\to0.
		\label{velocity-block-zero}
	\end{align}
	Since $X(s)-x^*\to0$, \eqref{velocity-block-zero} gives
	$X'(s)\to0$.  Returning to the original time and using
	$\dot x(t)=\sqrt{\varepsilon(t)}\,X'(s(t))$, we obtain
	\[
	\frac{\Vert \dot x(t)\Vert}{\sqrt{\varepsilon(t)}}
	=\Vert X'(s(t))\Vert\to0,
	\]
	which proves $\|\dot x(t)\|=o\bigl(\sqrt{\varepsilon(t)}\bigr)$. Moreover, \eqref{BG-zero} and $a(s(t))=1/{\varepsilon(t)}$ give
	\[
 	\mathcal{L}(x(t),\lambda^\dagger)-\mathcal{L}(x^*,\lambda^\dagger)=o(\varepsilon(t)).
	\]
	It remains to improve the convergence rates of objective residual and  feasibility violation from
	$O(\varepsilon(t))$ to $o(\varepsilon(t))$.
By the first condition in \eqref{bounded-conditions}, the function
\(K\) defined in \eqref{integrating-factor-K1} satisfies
\[
K(s)\geq\eta_1(s-s_T)\to\infty.
\]
Since \(\lambda^\dagger\) is constant, multiplying
\eqref{weighted-residual-ODE} by \(e^{K(s)}\), integrating over
\([s_T,s]\), and then integrating by parts yield
\begin{equation}
	\label{weighted-residual-little-o-representation}
	\begin{aligned}
		\Phi(s)={}&
		e^{-K(s)}\Phi(s_T)
		+\Lambda(s)-\lambda^\dagger
		-e^{-K(s)}
		\bigl(\Lambda(s_T)-\lambda^\dagger\bigr)\\
		&-\int_{s_T}^{s}
		e^{-(K(s)-K(u))}
		\left(\frac{\sigma}{\gamma}-p(u)\right)
		\bigl(\Lambda(u)-\lambda^\dagger\bigr)\,du.
	\end{aligned}
\end{equation}
By \eqref{dual-zero} and $K(s)\to\infty$, the first three terms on the
	right-hand side of
	\eqref{weighted-residual-little-o-representation} tend to zero. It remains to treat the integral term.   For any $\eta>0$, it follows from \eqref{dual-zero} that there exists
	$S_\eta\geq s_T$ such that
	\[
	\left\|\Lambda(u)-\lambda^\dagger\right\|\leq\eta
	\qquad\text{for all }u\geq S_\eta.
	\]
	For $u\in[s_T,S_\eta]$,  the first condition in \eqref{bounded-conditions} gives
	\[
	K(s)-K(u)
	\ge
	\eta_1(s-S_\eta).
	\]
	Consequently,
	\[
	\begin{aligned}
		&\left\|
		\int_{s_T}^{S_\eta}
		e^{-(K(s)-K(u))}
		\left(\frac{\sigma}{\gamma}-p(u)\right)
		\bigl(\Lambda(u)-\lambda^\dagger\bigr)\,du
		\right\|\\
		&\qquad\le
		e^{-\eta_1(s-S_\eta)}
		\int_{s_T}^{S_\eta}
		\left(\frac{\sigma}{\gamma}-p(u)\right)
		\Vert \Lambda(u)-\lambda^\dagger\Vert\,du
		\longrightarrow0.
	\end{aligned}
	\]
	For the remaining part, since
	$\sigma/\gamma-p(u)\ge0$,
	\[
	\begin{aligned}
		&\left\|
		\int_{S_\eta}^{s}
		e^{-(K(s)-K(u))}
		\left(\frac{\sigma}{\gamma}-p(u)\right)
		\bigl(\Lambda(u)-\lambda^\dagger\bigr)\,du
		\right\|\\
		&\qquad\le
		\eta
		\int_{S_\eta}^{s}
		e^{-(K(s)-K(u))}
		\left(\frac{\sigma}{\gamma}-p(u)\right)\,du=
		\eta
		\left[
		1-e^{-(K(s)-K(S_\eta))}
		\right]
		\le\eta.
	\end{aligned}
	\]
	Since $\eta>0$ is arbitrary, the integral term in
	\eqref{weighted-residual-little-o-representation} tends to zero.  Hence
	\[
	\Phi(s)\longrightarrow0.
	\]
	Using the definition of $\Phi$ and $a(s(t))=1/\varepsilon(t)$, we obtain
	\[
	\begin{aligned}
		\frac{\Vert Ax(t)-b\Vert}{\varepsilon(t)}
		=
		a(s(t))\Vert AX(s(t))-b\Vert=
		\frac{\sigma}{\gamma^2}\Vert \Phi(s(t))\Vert
		\longrightarrow0,\quad\text{i.e.}\quad \Vert Ax(t)-b\Vert=o(\epsilon(t)),
	\end{aligned}
	\]
	which together with $\mathcal{L}(x(t),\lambda^\dagger)-\mathcal{L}(x^*,\lambda^\dagger)=o(\varepsilon(t))$ and 
	\[
	f(x(t))-f(x^*)
	=
	\mathcal{L}(x(t),\lambda^\dagger)-\mathcal{L}(x^*,\lambda^\dagger)
	-\langle \lambda^\dagger,Ax(t)-b\rangle
	\]
	implies that $\vert f(x(t))-f(x^*)\vert=o(\varepsilon(t))$.
\end{proof}
\begin{remark}[Improvement from big-$O$ to little-$o$ rates]
	\label{little-o-comparison-remark}
	Most existing Nesterov-type primal--dual dynamics for linearly constrained
	convex optimization provide big-$O$ estimates; see
	\cite{He SICON 2021,Bot JDE 2021,Zeng TAC 2023,
		He Applied Analysis 2023}.
	Although \cite{Li JOTA 2025} obtains a little-$o$ estimate for a
	Lagrangian-type gap, the corresponding little-$o$ estimates for the
	objective residual and feasibility violation are not combined with the strong convergence to the minimum-norm solution under the same decay assumptions.
	More recently, He, Huang, Xiao, and Fang
	\cite{He arXiv 2026} established \(o(t^{-2})\) objective  residual and
	feasibility violation rates and an \(o(t^{-1})\) velocity rate for an unregularized
	primal--dual system whose primal and dual equations are both second order;
	its trajectory converges to an unspecified primal--dual solution. By contrast, under the same assumptions as
	Theorem~\ref{strong-theorem}, the Tikhonov-regularized mixed-order system
	\eqref{original-system} satisfies
	\[
		\mathcal{L}(x(t),\lambda^\dagger)-\mathcal{L}(x^*,\lambda^\dagger),\quad
	|f(x(t))-f(x^*)|,\quad
	\Vert Ax(t)-b\Vert=o(\varepsilon(t)),
	\quad
	\Vert\dot x(t)\Vert=o(\sqrt{\varepsilon(t)}),
	\]
	while its primal trajectory converges to the minimum-norm solution of the problem \eqref{problem} and its
	dual equation remains first order.  For
	\(\varepsilon(t)=c/t^2\), these estimates reduce to \(o(t^{-2})\) and
	\(o(t^{-1})\), respectively. Thus, under the same
	assumptions that guarantee the strong convergence in
	Theorem~\ref{strong-theorem}, the big-$O$ estimates of
	Theorem~\ref{fast-theorem} are improved to strict little-$o$ estimates.
\end{remark}

\section{Particular case: explicit Tikhonov coefficients}\label{particular case}
\label{explicit-coefficients-section}

In this section, we specialize the general results to the power law
\(\varepsilon(t)=c/t^r\), \(0<r\le2\), and to a logarithmically
modified coefficient \(\varepsilon(t)=c(\log t)^q/t^2\). In particular, for the critical choice
\(\varepsilon(t)=c/t^2\), which gives Nesterov-type damping, we obtain
the strong convergence to the minimum-norm solution without an eventual ball condition,
\(o(t^{-2})\) rates for the Lagrangian gap, objective residual, and
feasibility violation, and an \(o(t^{-1})\) velocity rate.
Throughout this section, set
\[
\beta_\infty^2
:=\limsup_{t\to\infty}\beta(t)^2<\infty.
\]

\subsection{Power-law coefficients
	\texorpdfstring{\(\varepsilon(t)=c/t^r\text{ with } 0<r\leq 2\)}{epsilon(t)=c/t power r}}

\begin{corollary}
	\label{power-law-corollary}
	Assume that $f:\mathbb{R}^n\to\mathbb{R}$ is convex and continuously
	differentiable, with an $L$-Lipschitz continuous gradient and the solution set $S\neq\emptyset$. Set
	\[
	\varepsilon(t)=\frac{c}{t^r},
	\qquad c>0,\qquad 0<r\le2.
	\]
	Suppose either \(0<r<2\), or \(r=2\) and
	\begin{equation}\label{critical-power-condition}
		\delta\sqrt c>3+L\beta_\infty^2.
	\end{equation}
Then \(\sigma>0\) can be chosen such that the following holds.
For every trajectory \((x,\lambda)\) of
\eqref{original-system} corresponding to this choice of \(\sigma\),
let \(\lambda^\dagger\) be the compatible KKT multiplier selected
in Lemma~\ref{KKT-lemma}, so that $\lambda(t_0)-\lambda^\dagger\in\operatorname{Ran}A.$
Then, as \(t\to\infty\),
	\begin{align}
		x(t)&\longrightarrow x^*,
		&\lambda(t)&\longrightarrow\lambda^\dagger,\nonumber\\
		\mathcal{L}(x(t), \lambda^\dagger)-\mathcal{L}(x^*,\lambda^\dagger)&=o(t^{-r}),
		&\Vert Ax(t)-b\Vert&=o(t^{-r}),\nonumber\\
		\left|f(x(t))-f(x^*)\right|&=o(t^{-r}),
		&\Vert\dot x(t)\Vert&=o(t^{-r/2}).\nonumber
	\end{align}
\end{corollary}

\begin{proof}
	A direct calculation gives \(s(t)\to\infty\) and
	\begin{equation*}\label{power-p}
		p(s(t))=\frac{r}{\sqrt c}\,t^{r/2-1},
		\qquad
		p'(s(t))=-\frac{r(2-r)}{2c}\,t^{r-2}.
	\end{equation*}
	If \(0<r<2\), then \(p(s)\to0\) and
	\((-p'(s))_+\to0\).  Choose
	\[
	0<h<\frac{\delta}{1+L\beta_\infty^2/2},
	\qquad \sigma:=\gamma h.
	\]
	Then
	\[
	\frac{\sigma}{\gamma}-p(s)\to h>0,
	\]
	\[
	\liminf_{s\to\infty}
	\left[
	\delta-\frac{\sigma}{\gamma}-\frac{p(s)}2
	-\frac{\sigma L}{2\gamma}\beta(t(s))^2
	\right]
	\ge
	\delta-h\left(1+\frac L2\beta_\infty^2\right)>0.
	\]
	Moreover,
	\[
	\frac{\sigma}{4\gamma}(-p'(s))_+
	=
	\frac{h}{4}(-p'(s))_+
	\to0,
	\qquad
	\frac{\sigma/\gamma+p(s)}8
	=
	\frac{h+p(s)}8
	\to\frac h8>0.
	\]
	Hence, the third condition in
	\eqref{bounded-conditions} also holds for all sufficiently large \(s\). Thus, there exist \(\eta_1,\eta_2>0\) such that
	\eqref{bounded-conditions} holds for all sufficiently large \(s\).
	Since \(a(s)\to\infty\), the remaining requirements of
	Proposition~\ref{boundedness-proposition} are satisfied after
	increasing \(s_T\) if necessary. Therefore,
	Theorems~\ref{fast-theorem}, \ref{strong-theorem}, and
	\ref{little-o-theorem} apply.

	If \(r=2\), then \(p(s)=2/\sqrt c\) and \(p'(s)=0\).  Under
	\eqref{critical-power-condition}, choose
	\[
	0<h<
	\frac{
		\delta-3/\sqrt c-(L/\sqrt c)\beta_\infty^2
	}{1+L\beta_\infty^2/2},
	\qquad
	\sigma:=\gamma\left(\frac2{\sqrt c}+h\right).
	\]
	The first pointwise expression equals \(h\), the second has positive
	lower limit
	\[
	\delta-\frac3{\sqrt c}-\frac{L}{\sqrt c}\beta_\infty^2
	-h\left(1+\frac L2\beta_\infty^2\right),
	\]
	and the third condition in	\eqref{bounded-conditions} is immediate from \(p'=0\).  The same three
	theorems therefore apply.  Substituting \(\varepsilon(t)=ct^{-r}\) in
	their conclusions gives the stated results.
\end{proof}

\begin{remark}
	\label{critical-coefficient-remark}
	For \(\varepsilon(t)=c/t^2\), the damping coefficient is
	\(\delta\sqrt c/t\), so \eqref{original-system} is of Nesterov type.
	Corollary~\ref{power-law-corollary} gives, without an eventual
	inside/outside-ball condition,
	\[
	x(t)\to x^*,\quad
	\mathcal{L}(x(t), \lambda^\dagger)-\mathcal{L}(x^*,\lambda^\dagger),\ \Vert Ax(t)-b\Vert,\
	|f(x(t))-f(x^*)|=o(t^{-2}),\qquad
	\Vert\dot x(t)\Vert=o(t^{-1}).
	\]
	Existing minimum-norm strong-convergence results without such a ball
	condition mainly concern constant damping or \(\alpha/t^q\) with
	\(0<q<1\), and hence do not cover the critical damping \(\alpha/t\);
	see \cite{Battahi ASVAO 2025,Csetnek JEE 2026}.
	Moreover, Attouch and L\'aszl\'o \cite{Attouch MMOR 2024} identified the
	corresponding unconstrained strong-convergence problem as open.  The case
	\(A=0\), \(b=0\), and \(\beta=0\) therefore provides a finite-dimensional
	answer, while the full result extends it to linearly constrained
	optimization.
\end{remark}

\subsection{The case
	\texorpdfstring{\(\varepsilon(t)=c(\log t)^q/t^2\)}
	{epsilon(t)=c(log t)^q/t^2}}
\begin{corollary}
	\label{log-relaxed-corollary}
	Assume that $f:\mathbb{R}^n\to\mathbb{R}$ is convex and continuously
	differentiable, with an $L$-Lipschitz continuous gradient and the solution set $S\neq\emptyset$. Set
	\begin{equation}\label{log-relaxed-epsilon}
		\varepsilon(t)=\frac{c(\log t)^q}{t^2},
		\qquad c>0,\qquad q>0,\qquad t_0>e^{q/2}.
	\end{equation}
Then \(\sigma>0\) can be chosen such that the following holds.
For every trajectory \((x,\lambda)\) of
\eqref{original-system} corresponding to this choice of \(\sigma\),
let \(\lambda^\dagger\) be the compatible KKT multiplier selected
in Lemma~\ref{KKT-lemma}, so that $\lambda(t_0)-\lambda^\dagger\in\operatorname{Ran}A.$
Then, as \(t\to\infty\),
	\begin{align}
		x(t)&\longrightarrow x^*,
		&\lambda(t)&\longrightarrow\lambda^\dagger,\nonumber\\
		\mathcal{L}(x(t), \lambda^\dagger)-\mathcal{L}(x^*,\lambda^\dagger)
		&=o\left(\frac{(\log t)^q}{t^2}\right),
		&\Vert Ax(t)-b\Vert
		&=o\left(\frac{(\log t)^q}{t^2}\right),\nonumber\\
		\left|f(x(t))-f(x^*)\right|
		&=o\left(\frac{(\log t)^q}{t^2}\right),
		&\Vert\dot x(t)\Vert
		&=o\left(\frac{(\log t)^{q/2}}{t}\right).\nonumber
	\end{align}
\end{corollary}

\begin{proof}
	The assumption on \(t_0\) makes \(\varepsilon\) decreasing.  Moreover,
	\[
	s(t)=\frac{\sqrt c}{1+q/2}
	\left[(\log t)^{1+q/2}-(\log t_0)^{1+q/2}\right]
	\longrightarrow\infty,
	\]
	and a direct calculation gives
	\begin{equation*}\label{log-relaxed-p}
		\begin{aligned}
			p(s(t))
			&=\frac1{\sqrt c}
			\left(2-\frac q{\log t}\right)(\log t)^{-q/2}
			\longrightarrow0,\\
			p'(s(t))
			&=\frac q{c}
			\left[-(\log t)^{-q-1}
			+\left(\frac q2+1\right)(\log t)^{-q-2}\right]
			\longrightarrow0.
		\end{aligned}
	\end{equation*}
	Choose
	\[
	0<h<
	\frac{\delta}{1+L\beta_\infty^2/2},
	\qquad
	\sigma:=\gamma h.
	\]
	Then
	\[
	\frac{\sigma}{\gamma}-p(s)\to h>0,
	\]
	and
	\[
	\liminf_{s\to\infty}
	\left[
	\delta-\frac{\sigma}{\gamma}-\frac{p(s)}2
	-\frac{\sigma L}{2\gamma}\beta(t(s))^2
	\right]
	\geq
	\delta-h\left(1+\frac L2\beta_\infty^2\right)>0.
	\]
	Moreover,
	\[
	\frac{\sigma}{4\gamma}(-p'(s))_+\to0,
	\qquad
	\frac{\sigma/\gamma+p(s)}8\to\frac h8>0.
	\]
	Hence, all three conditions in
	\eqref{bounded-conditions} hold for all sufficiently large \(s\).
	Theorems~\ref{fast-theorem}, \ref{strong-theorem}, and
	\ref{little-o-theorem} therefore apply, and substitution of
	\eqref{log-relaxed-epsilon} gives the conclusions.
\end{proof}

\begin{remark}[Rate--damping trade-off]
	The critical choice \(\varepsilon(t)=c/t^2\) gives the sharper rates
	\(o(t^{-2})\) and \(o(t^{-1})\), but requires
	\(\delta\sqrt c>3+L\beta_\infty^2\).  The choice
	\(c(\log t)^q/t^2\) admits every \(\delta>0\), at the price of the factors
	\((\log t)^q\) and \((\log t)^{q/2}\) in the corresponding rates.  Also,
	since
	\[
	\frac{d}{dt}\frac1{\sqrt{\varepsilon(t)}}
	=\frac12p(s(t)),
	\]
	eventual boundedness of \(p\) implies
	\(\varepsilon(t)\ge(C+Mt/2)^{-2}\) for suitable \(C,M>0\).  Consequently, the first pointwise condition
	\[
	\frac{\sigma}{\gamma}-p(s)\ge\eta_1>0
	\]
	imposed in Theorems~\ref{fast-theorem} and~\ref{strong-theorem}
	precludes \(\varepsilon(t)=o(t^{-2})\).  Hence \(t^{-2}\) is the
	fastest decay order of the Tikhonov coefficient allowed by the
	hypotheses of the present convergence theory.
\end{remark}

\section{Conclusion}
\label{conclusion-section}

A Tikhonov-regularized mixed-order primal--dual dynamical system with
implicit Hessian damping was introduced for linearly constrained convex
optimization. A nonlinear change of time revealed a dissipative
structure suitable for Lyapunov analysis. Under a single set of
parameter conditions, the primal trajectory was shown to converge
strongly to the minimum-norm solution, without an eventual
inside/outside-ball condition, while the multiplier converges to a
compatible KKT multiplier. Under the same conditions, the Lagrangian
gap, objective residual, and feasibility violation decay as
\(o(\varepsilon(t))\), and the velocity satisfies
\(o(\sqrt{\varepsilon(t)})\). For the critical coefficient
\(\varepsilon(t)=c/t^2\), these estimates become \(o(t^{-2})\) and
\(o(t^{-1})\), respectively. The numerical illustration further
demonstrated the oscillation-attenuation effect of the implicit Hessian
term under different curvature levels. Future work may address structure-preserving discretizations and
extensions to nonsmooth or composite linearly constrained problems.
	
	%%=======================================================================
	
	%	
\end{document}